\documentclass[11pt,reqno]{amsproc}
\usepackage[letterpaper,margin=1in]{geometry}
\usepackage{amsmath,amssymb,amsfonts,amsthm,mathrsfs}
\usepackage{enumitem}
\usepackage{tikz}
\usepackage{hyperref}
\hypersetup{colorlinks=true,linkcolor=blue,citecolor=magenta,urlcolor=teal}
\allowdisplaybreaks
\numberwithin{equation}{section}
\newtheorem{theorem}{Theorem}[section]
\newtheorem{proposition}[theorem]{Proposition}
\newtheorem{lemma}[theorem]{Lemma}
\newtheorem{corollary}[theorem]{Corollary}
\newtheorem{claim}[theorem]{Claim}
\theoremstyle{remark}
\newtheorem{remark}[theorem]{Remark}
\newcommand{\p}{\partial}
\newcommand{\RR}{\mathbb R}
\newcommand{\norm}[1]{\left\|#1\right\|}
\newcommand{\abs}[1]{\left|#1\right|}
\newcommand{\eps}{\varepsilon}
\newcommand{\les}{\lesssim}
\newcommand{\curl}{\operatorname{curl}}
\renewcommand{\div}{\operatorname{div}}

\newcommand{\esssup}{\operatorname*{ess\,sup}}
\newcommand{\Caxi}{\mathsf{C}_{\mathrm{axi}}}

\title[Asymptotically axisymmetric Navier-Stokes with analytic forcing]{Regularity of asymptotically axisymmetric solutions to the 3D Navier-Stokes equations with analytic forcing}

\author{Peter Constantin} 
\address{Department of Mathematics, Princeton University, Princeton, NJ 08540}
\email{const@math.princeton.edu}
\urladdr{https://web.math.princeton.edu/~const/}

\author{Mihaela Ignatova}
\address{Department of Mathematics, Temple University, Philadelphia, PA 19122}
\email{ignatova@temple.edu}
\urladdr{https://sites.temple.edu/ignatova/}

\author{Vlad Vicol}
\address{Department of Mathematics, Courant Institute, New York University, New York, NY, 10012}
\email{vicol@cims.nyu.edu}
\urladdr{https://cims.nyu.edu/~vicol/}

\subjclass[2020]{Primary 35Q30; Secondary 35B44, 35Q31}
\keywords{Navier-Stokes equations, local regularity, axial symmetry, analytic forcing.}
\date{}

\begin{document}
\begin{abstract}
OpenAI~\cite{OpenAIManuscript} has recently announced a proof of finite time singularity formation for the 3D Navier-Stokes equations, in the presence of a \emph{$C^\infty$-smooth body force}. The construction in~\cite{OpenAIManuscript} has a few key features, among which we single out: (i) the angular mean of the solution satisfies specific Type II bounds which are anisotropic; (ii) the solution is exactly axisymmetric in a collapsing core region. In this paper we consider solutions of the 3D Navier-Stokes equations in the presence of a \emph{real-analytic body force}, and assume that these solutions satisfy properties (i) and (ii) above. We prove that such solutions are in fact regular at the putative singular point. As a consequence, in the construction of~\cite{OpenAIManuscript}, and in any construction with properties (i) and (ii) whose force remains bounded in $C^2$ up to the singular time, the force \emph{can neither vanish} identically near the singular point, \emph{nor be real analytic} in the space variables, locally uniformly in time. The main idea of the proof is inspired by our earlier work~\cite{CIV26} and the companion paper~\cite{CIVEulerLength}: zooming-in at the putative singularity using the anisotropic length scales provided by the a priori bounds, we arrive at ancient limits whose PDE evolution imposes additional rigidity.
\end{abstract}
\maketitle

\section{Introduction}
\label{sec:aniso:introduction}

In a recent manuscript~\cite{OpenAIManuscript}, OpenAI has announced a proof of finite time singularity formation for the 3D Navier-Stokes equations in the presence of a $C^\infty$-smooth forcing. The main theorem of the manuscript~\cite[Theorem~1.1]{OpenAIManuscript} asserts that for every viscosity $\nu>0$ there exists a force $f\in C^\infty_c(\RR^3\times(0,\infty))$ and a smooth solution $(u,\pi)$ of the 3D Navier-Stokes equations on $\RR^3\times[0,1)$
\begin{align*}
 \p_tu+(u\cdot\nabla)u-\nu\Delta u+\nabla\pi&=f \,,\\
 \div u&=0 \,, \\
 u(\cdot,0)&=0 \,,
\end{align*}
which has a fixed compact spatial support, finite kinetic energy $\sup_{0\le t<1}\|u(\cdot,t)\|_{L^2}<\infty$, and the solution blows up as $t\uparrow 1$:
\[
\limsup_{t\uparrow1}\norm{u(\cdot,t)}_{L^\infty}=\infty
\,.
\]
We refer to the solution built in~\cite{OpenAIManuscript} as the \emph{OpenAI construction}. The singular point is $(x,t)=(0,1)$, and it is isolated: the velocity diverges along points converging to the origin~\cite[equations~(10.20)--(10.21)]{OpenAIManuscript}, while every space-time derivative of the velocity has a one-sided limit at the blowup time, on compact sets avoiding the origin; see~\cite[Theorem~3.1(ii), equation~(3.5), and the proof of Lemma~10.2]{OpenAIManuscript}. The force $f$, in turn, is defined as the residual of the momentum equation, left over after the construction's ansatz, and it is $C^\infty$ smooth but not real analytic: neither on $\RR^3$, because it is compactly supported, nor in the vicinity of the singular point (see Section~\ref{sec:properties:criteria}). In Appendix~\ref{app:scales} we record those properties of the OpenAI construction which are relevant to this paper; they are deduced from statements of~\cite{OpenAIManuscript}, which we cite with their locations.

Of the many features of the solution constructed in~\cite{OpenAIManuscript}, we single out the two which enter the hypotheses of our main theorem:
\begin{itemize}[leftmargin=2em]
\item[(i)] Near the singular point the angular mean $v=\mathcal Pu$ of the velocity, defined in~\eqref{eq:interior:average} below, satisfies Type~II bounds which are anisotropic: its radial component $v_r$ is bounded by $(1-t)^{-1/2}$, while its axial and azimuthal components $v_z$ and $v_\theta$ are bounded by $(1-t)^{-1/2-h}$ for some $0<h<1/100$. Moreover, each $r$-derivative costs a factor of $(1-t)^{-1/2}$, while each $z$-derivative costs only a factor of $(1-t)^{-1/2+h}$.
\item[(ii)] The full velocity field $u$ is exactly axisymmetric on a collapsing core region $B(c\sqrt{1-t})$, for some $c>0$ fixed.
\end{itemize}
We do not claim to have verified the correctness of the construction in~\cite{OpenAIManuscript}: Appendix~\ref{app:scales} records the statements of that manuscript on which we rely, with their locations, and deduces from them the properties used in this paper.

\subsection{Problem setup}
\label{sec:aniso:setup}

The solution constructed in~\cite{OpenAIManuscript} is defined for $t\in[0,1)$ and becomes singular as $t\uparrow1$. Since the equations are invariant under translations in time, we place the singular time at $t=0$, which is more convenient: we replace the triple $(u,\pi,f)$ of the manuscript~\cite{OpenAIManuscript} by the shifted functions
\begin{align}
(u,\pi,f)(x,t)\longmapsto(u,\pi,f)(x,t+1)
\,,\qquad -1\le t<0
\,.
\label{eq:time:shift}
\end{align}
We denote the shifted triple again by $(u,\pi,f)$. The singular point is then $(x,t)=(0,0)$, and the time to the singularity is $-t$, replacing the $1-t$ in the time variable of~\cite{OpenAIManuscript}. In particular, the factor $(1-t)$ of item~(i) becomes $(-t)$, and the core of item~(ii) becomes $B(c\sqrt{-t})$. From here on, every statement of this paper is in the shifted time coordinate of~\eqref{eq:time:shift}. 

The goal of this paper is to analyze solutions of the 3D Navier-Stokes equations which are \emph{a priori} assumed to obey properties (i) and (ii) above. We prove that if such solutions were to arise from the forced problem with a body force which is \emph{real analytic} in the space variables, in the sense made precise in Theorem~\ref{thm:main} below, then $(x,t)=(0,0)$ is a regular point.\footnote{Recall that a point $(x_0,t_0)$ with $x_0\in B(1)$ and $-1<t_0\le0$ is called \emph{regular} if $\esssup_{B(x_0,\rho)\times(t_0-\rho^2,t_0)}|u|<\infty$ for some $\rho>0$ with $B(x_0,\rho)\times(t_0-\rho^2,t_0)\subset B(0,1)\times(-1,0)$, and \emph{singular} otherwise. Here and throughout this paper, we use the notation $B(x,\rho)=\{y\in\RR^3:|y-x|<\rho\}$ and $B(\rho)=B(0,\rho)$.} 

Throughout the paper we consider suitable weak solutions, in the sense of Caffarelli, Kohn, and Nirenberg~\cite{CKN}, of the forced 3D Navier-Stokes equations
\begin{subequations}
\label{eq:nse:forced}
\begin{align}
\p_tu+(u\cdot\nabla)u-\Delta u+\nabla\pi=f\,,
\\
\div u=0\,,
\end{align}
\end{subequations}
where $u\colon\mathcal Q\to\RR^3$ is the velocity, $\pi\colon\mathcal Q\to\RR$ is the pressure, and $f\colon\mathcal Q\to\RR^3$ is the force.\footnote{\label{foot:aniso:pressure}The partial regularity results of~\cite{CKN} are stated for a divergence-free force. We apply them to $(u,\pi-\pi_f)$ with the force $f-\nabla\pi_f$, where $\Delta\pi_f=\div f$ on $B(1)$ with zero Dirichlet data. The function $\nabla\pi_f$ is bounded on $B(1)$ uniformly in time, so that $f-\nabla\pi_f$ is bounded and divergence free, and the local energy inequality is unchanged, since $\int u\cdot\nabla(\pi_f\varphi)\,dx=0$ for $\varphi\in C^\infty_0(B(1))$. Since $\curl(f-\nabla\pi_f)=\curl f$, and since $\pi_f$ is axisymmetric whenever $f$ is, so that $(f-\nabla\pi_f)_\theta=f_\theta$, the derivative estimates, the circulation equation, and the vorticity equations of Sections~\ref{sec:aniso:identities}--\ref{sec:aniso:closure}, which are where the $C^2$ bound~\eqref{eq:interior:force:c-two} on the force enters, are written with $f$ itself.} Here $\mathcal Q:=B(1)\times(-1,0)$ is the unit parabolic cylinder,\footnote{The properties of the construction in~\cite{OpenAIManuscript} are supplied on a cylinder $B(R)\times(-\delta,0)$ for some $R,\delta>0$; see Appendix~\ref{app:scales}. Upon decreasing $R$ or $\delta$ we may take $\delta=R^2\le1$, and the scaling $u\mapsto Ru(R\,\cdot,R^2\,\cdot)$, $\pi\mapsto R^2\pi(R\,\cdot,R^2\,\cdot)$, $f\mapsto R^3f(R\,\cdot,R^2\,\cdot)$ carries the hypotheses of Theorem~\ref{thm:main} from $B(R)\times(-R^2,0)$ to the unit parabolic cylinder $\mathcal Q$, with constants changed by fixed powers of $R$.} whose upper time endpoint $t=0$ is the blowup time of the construction.

Unless stated otherwise, we take $(u,\pi)$ to be a suitable weak solution of~\eqref{eq:nse:forced} which is \emph{smooth on compact subsets} of $\mathcal Q$, and hence $f$ is smooth there as well. Note that the definition of a suitable weak solution on $\mathcal Q$ in the sense of~\cite{CKN}, as written in~\cite{Lin98,LadyzhenskayaSeregin99}, includes the integrability requirement
\begin{align}
u\in L^\infty_tL^2_x(\mathcal Q)\cap L^2_tH^1_x(\mathcal Q)
\,,\qquad
\pi\in L^{3/2}_{x,t}(\mathcal Q)
\,.
\label{eq:interior:energy:class}
\end{align}
We assume no regularity of the solution $(u,\pi)$ at time $t=0$.

\subsection{Notation}
\label{sec:aniso:notation}

The solutions of~\eqref{eq:nse:forced} need not be axisymmetric, but it is convenient to introduce the notation attached to the cylindrical coordinate system on $\RR^3$. For $\varphi\in[0,2\pi)$, denote by $Q_\varphi$ the rotation through angle $\varphi$ about the $z$-axis. For a solution $u\colon\mathcal Q\to\RR^3$ of~\eqref{eq:nse:forced} we define the rotation and the projection operators\footnote{\label{foot:angular:mean}Since rotations about the $z$-axis preserve $B(1)$, the angular mean projection operator $\mathcal P$ is well defined on $\mathcal Q$.}
\begin{subequations}
\label{eq:interior:average}
\begin{align}
(\mathscr R_\varphi u)(x,t)
&=Q_\varphi u(Q_\varphi^{-1}x,t)\,,
\qquad
\mathcal Pu=\frac1{2\pi}\int_0^{2\pi}\mathscr R_\varphi u\,d\varphi\,,
\end{align}
and denote
\begin{align}
v=\mathcal Pu\,,\qquad w=u-v\,,\qquad \mathcal Pw=0
\,.
\end{align}
\end{subequations}
We call $v$ the \emph{angular mean}, or the \emph{axisymmetric part}, of $u$, and we refer to $w$ as the \emph{non-axisymmetric part}. The field $u$ is \emph{axisymmetric} if $w=0$. With $x=(x',z)$, $r=|x'|$, and the usual cylindrical frame $(e_r,e_\theta,e_z)$, an axisymmetric field has the form $u=u_re_r+u_\theta e_\theta+u_ze_z$, where $u_r,u_\theta,u_z$ depend only on $(r,z,t)$. For a general field $u$, the cylindrical coefficients $u_r,u_\theta,u_z$ depend on $\theta$ as well, and the cylindrical coefficients of its angular mean $v=\mathcal Pu$ are their averages in $\theta$, that is, $v_\circ=\frac1{2\pi}\int_0^{2\pi}u_\circ\,d\theta$ for $\circ \in\{r,\theta,z\}$.

\subsection{Main result}
\label{sec:aniso:main}
Our main result is the following.
\begin{theorem}[{\bf Regularity under anisotropic Type~II bounds and analytic forcing}]
\label{thm:main}
Let $(u,\pi)$ be a suitable weak solution of~\eqref{eq:nse:forced} on $\mathcal Q$, smooth on compact subsets of $\mathcal Q$, and let $0<h<1/2$. Assume that the force is bounded in $C^2$ up to the blowup time,
\begin{align}
M_f:=\sup_{-1<t<0}\norm{f(\cdot,t)}_{C_x^2(B(1))}<\infty
\,.
\label{eq:interior:force:c-two}
\end{align}
Assume that the force is real analytic in the space variables, locally uniformly on cylinders compactly contained in $\mathcal Q$: for every $R<1$ and every compact interval $J\subset(-1,0)$ there are $M_{R,J}<\infty$ and $a_{R,J}>0$ such that
\begin{align}
\sup_{t\in J}\norm{\p_x^\alpha f(\cdot,t)}_{L^\infty(B(R))}
\le M_{R,J}a_{R,J}^{-|\alpha|}|\alpha|!
\qquad\mbox{for all }\alpha\in\mathbb N_0^3
\,.
\label{eq:interior:force:analytic}
\end{align}
Assume that the velocity has the following two properties.
\begin{enumerate}[label=(\roman*),leftmargin=2em]
\item Its angular mean $v:=\mathcal Pu$ satisfies, on the fixed Euclidean ball $B(1)$, for $-1<t<0$, and all nonnegative integers $a,b$ with $a+b\le2$,
\begin{subequations}
\label{eq:interior:mean:bounds}
\begin{align}
\|\p_r^a\p_z^b v_r(\cdot,t)\|_{L^\infty(B(1))}
&\le \Caxi(-t)^{-1/2-a/2-(1/2-h)b}\,,
\\
\| \p_r^a\p_z^b v_z(\cdot,t)\|_{L^\infty(B(1))}
+\|\p_r^a\p_z^b v_\theta(\cdot,t)\|_{L^\infty(B(1))}
&\le \Caxi(-t)^{-1/2-h-a/2-(1/2-h)b}
\,,
\end{align}
\end{subequations}
for a constant $\Caxi>0$ which is independent of $t$.
\item For each $t\in(-1,0)$ there is a radius $\rho(t)\in(0,1)$ such that the solution $u(\cdot,t)$ is exactly axisymmetric on the ball of radius $\rho(t)$; equivalently,
\begin{align}
w(\cdot,t)=0
\qquad\mbox{on }B(\rho(t))
\,.
\label{eq:interior:core}
\end{align}
No regularity or quantitative lower bound is assumed on $\rho$; we only need that $\rho(t)>0$ at all $t\in(-1,0)$.
\end{enumerate}
Then $u=v$ on $\mathcal Q$, and $(0,0)$ is a regular point: there exist $r_*,\delta_*>0$ such that
\begin{align}
\sup_{B(r_*)\times(-\delta_*,0)}|u|<\infty
\,.
\label{eq:interior:regular}
\end{align}
\end{theorem}

\begin{remark}
\label{rem:main:comments}
A few comments on the hypotheses of Theorem~\ref{thm:main} are in order:
\begin{itemize}[leftmargin=2em]
\item[(a)] The force assumptions~\eqref{eq:interior:force:c-two} and~\eqref{eq:interior:force:analytic} have distinct roles. Spatial analyticity~\eqref{eq:interior:force:analytic} is used at each fixed time before the blowup time, so the analyticity radius $a_{R,J}$ is allowed to tend to zero as $J$ approaches time $0$.\footnote{Condition~\eqref{eq:interior:force:analytic} holds if and only if, for every $R<1$ and every compact $J\subset(-1,0)$, each $f(\cdot,t)$ with $t\in J$ is the restriction of a holomorphic function on a fixed complex neighborhood of $\overline{B(R)}$, bounded uniformly in $t\in J$.} The regularity argument of Sections~\ref{sec:aniso:identities}--\ref{sec:aniso:closure}, which proves Theorem~\ref{thm:aniso:main}, uses instead the uniform spatial $C^2$ bound~\eqref{eq:interior:force:c-two}. A force which is real analytic in $(x,t)$ on a neighborhood of $\overline{B(1)}\times[-1,0]$ satisfies both conditions, and so does $f=0$. No symmetry of $f$ is assumed, although, together with~\eqref{eq:interior:core}, the hypothesis~\eqref{eq:interior:force:analytic} does make $\curl f$ rotation invariant on all of $\mathcal Q$; see Remark~\ref{rem:force:symmetry}.
\item[(b)] The bounds~\eqref{eq:interior:mean:bounds} concern only the angular mean of the solution. Apart from the $L^\infty_tL^2_x\cap L^2_tH^1_x(\mathcal Q)$ regularity implicit in the definition of a suitable weak solution, no bound on $w$ or its derivatives is assumed outside the core. The OpenAI construction does satisfy the bounds in~\eqref{eq:interior:mean:bounds}; see Section~\ref{sec:properties:values}.
\item[(c)] The OpenAI construction satisfies~\eqref{eq:interior:core} with $\rho(t)=c\sqrt{-t}$, for $c>0$ a constant such that $c<\sqrt{2X_a}$; see Section~\ref{sec:properties:structure}.
\item[(d)] The force used in~\cite{OpenAIManuscript} does not satisfy~\eqref{eq:interior:force:analytic}. This is consistent with the construction, whose force is built from compactly supported cutoffs.
\item[(e)] Exact axisymmetry on the core~\eqref{eq:interior:core} and analyticity of the force~\eqref{eq:interior:force:analytic} may be traded for an upper bound on the non-axisymmetric part of the velocity field. If $\|w(\cdot,t)\|_{C^3(B(1))}$ is bounded uniformly in $t\in (-1,0)$, then $(0,0)$ is regular. This statement requires no analyticity assumption at all; see Remark~\ref{rem:interior:approximate}. The OpenAI construction does not satisfy such a bound on $w$; see Section~\ref{sec:properties:criteria}.
\end{itemize}
\end{remark}

\subsection{From analytic forcing and axisymmetric core to an axisymmetric problem}
\label{sec:aniso:axisymmetric}

Our first observation is that the analyticity assumption~\eqref{eq:interior:force:analytic} on the force makes $u(\cdot,t)$ real analytic on $B(1)$, for every fixed $t\in(-1,0)$. This follows from the classical interior analyticity theorem for the Navier-Stokes system, due to Kahane~\cite[Theorem~1.2 and p.~387]{Kahane69}; we recall this result as Theorem~\ref{thm:analytic:interior}, and Corollary~\ref{lem:analytic:slice} is the form in which we apply it. Combined with the axisymmetric core assumption~\eqref{eq:interior:core}, this makes $u$ axisymmetric on all of $\mathcal Q$: for every rotation $Q_\varphi$ about the $z$-axis, the analytic field $u-\mathscr R_\varphi u$ vanishes on the core, hence on all of $B(1)$ by the identity theorem; these details are given in Section~\ref{sec:reduction:core}. Thus $u=v$, and the bounds on the angular mean in~\eqref{eq:interior:mean:bounds} become bounds on the full velocity field. Once $u$ is known to be axisymmetric, we average~\eqref{eq:nse:forced} over the rotations $\mathscr R_\varphi$, which leaves the velocity terms unchanged and replaces only the pressure and the force by their angular means. By~\eqref{eq:interior:force:c-two}, the angle-averaged force $\mathcal Pf$ remains bounded in $C^2$.

This reduces Theorem~\ref{thm:main} to the regularity of axisymmetric solutions of the 3D Navier-Stokes equations with $C^2$ forcing, under the anisotropic bounds inherited from~\eqref{eq:interior:mean:bounds}. That result is Theorem~\ref{thm:aniso:main} below, which is of independent interest. In order to state it, we fix the standing assumptions of Sections~\ref{sec:aniso:identities}--\ref{sec:aniso:closure}: $(u,\pi)$ is an axisymmetric suitable weak solution of the forced equations~\eqref{eq:nse:forced} on $\mathcal Q$, which is smooth on compact subsets of $\mathcal Q$,\footnote{Smoothness in Cartesian coordinates gives $u_r=u_\theta=\p_ru_z=0$ at $r=0$, for all $t<0$.} in the class described in Section~\ref{sec:aniso:setup}, and whose force satisfies~\eqref{eq:interior:force:c-two}. For $0<h<1/2$ we define the radial and axial lengths as in~\cite{OpenAIManuscript}:
\begin{align}
\ell_r(t)=(-t)^{1/2}\,,\qquad \ell_z(t)=(-t)^{1/2-h}
\,.
\label{eq:aniso:lengths}
\end{align}
The range $0<h<1/2$ ensures that both lengths tend to zero as $t\uparrow0$. The bounds in~\eqref{eq:aniso:bounds} below are imposed on the scalar cylindrical coefficients of $u$ and on their derivatives through total order two, with each radial derivative costing a factor of $\ell_r^{-1}$, and each axial derivative costing a factor of $\ell_z^{-1}$. The following theorem is the anisotropic counterpart of the main result of~\cite{CIVEulerLength}:

\begin{theorem}[{\bf Regularity of axisymmetric solutions under anisotropic bounds}]
\label{thm:aniso:main}
Let $(u,\pi)$ be an axisymmetric suitable weak solution of~\eqref{eq:nse:forced} in $\mathcal Q$, which is smooth on compact subsets of $\mathcal Q$, with a force satisfying~\eqref{eq:interior:force:c-two}. Let $0<h<1/2$. Suppose that there exists a constant $\Caxi>0$ such that
\begin{subequations}
\label{eq:aniso:bounds}
\begin{align}
|\p_r^a\p_z^b u_r(r,z,t)|
&\le \Caxi\,(-t)^{-1/2-a/2-(1/2-h)b}\,,
\\*
|\p_r^a\p_z^b u_z(r,z,t)|+|\p_r^a\p_z^b u_\theta(r,z,t)|
&\le \Caxi\,(-t)^{-1/2-h-a/2-(1/2-h)b}
\,,
\end{align}
\end{subequations}
on $\mathcal Q$, for all nonnegative integers $a,b$ with $a+b\le2$. Then, $(0,0)$ is a regular point.
\end{theorem}

\begin{remark}
\label{rem:aniso:comments}
We make a few comments on the hypotheses of Theorem~\ref{thm:aniso:main}.
\begin{itemize}[leftmargin=2em]
\item[(a)] The bounds in~\eqref{eq:aniso:bounds} say that the rescaled fields $(-t)^{1/2}u_r$, $(-t)^{1/2+h}u_z$, and $(-t)^{1/2+h}u_\theta$, regarded as functions of $(r/\ell_r(t),z/\ell_z(t))$, are bounded in $C^2$ uniformly in $t$; see~\eqref{eq:aniso:zoom:derivatives} below.
\item[(b)] Neither~\eqref{eq:aniso:bounds} nor the Type~I bound $|u|\le C(-t)^{-1/2}$ implies the other: at $a=b=0$ the bounds on $u_z$ and $u_\theta$ in~\eqref{eq:aniso:bounds} are weaker by the factor $(-t)^{-h}$, while at every $b\ge1$ the bound imposed on $u_r$ is smaller by the factor $(-t)^{hb}$ than the bound $(-t)^{-1/2-(a+b)/2}$ which the parabolic scaling attaches to a Type~I velocity bound. Theorem~\ref{thm:aniso:main} is therefore not contained in the Type~I results of Chen, Strain, Yau, and Tsai~\cite{CSTY08}, Chen, Strain, Tsai, and Yau~\cite{CSTY09}, Koch, Nadirashvili, Seregin, and \v{S}ver\'ak~\cite{KNSS09}, Seregin and \v{S}ver\'ak~\cite{SereginSverak09}, or~Zhang~\cite{QiZhang26}, nor are those results contained in it.
\item[(c)] The proof of Theorem~\ref{thm:aniso:main} may be completed using only the bounds on $u_r$ and $u_z$ from~\eqref{eq:aniso:bounds}, with the assumptions on $u_\theta$ dropped altogether; see Remark~\ref{rem:interior:meridional}. For cosmetic reasons, we nevertheless state~\eqref{eq:aniso:bounds} for the three components alike, and do not pursue the minimal hypothesis.
\end{itemize}
\end{remark}

Next, we outline the proof of Theorem~\ref{thm:aniso:main}; the estimates and the full proof are postponed to Sections~\ref{sec:aniso:identities}--\ref{sec:aniso:closure}. The strategy is that of the companion paper~\cite{CIVEulerLength}, whose steps are as follows. In~\cite{CIVEulerLength} we consider unforced axisymmetric solutions which are bounded in $C^2$ after rescaling \emph{isotropically} at a single vanishing length scale $\ell(t)$, for which $(-t)/\ell(t)^2$ tends to zero. Using this length scale we zoom in, along times and points at which the gradient of the meridional velocity is not small compared with $(-t)^{-1}$. Since the pressure is not controlled after such a zoom, the limit is taken in the vorticity equations rather than in the momentum equation. Partial regularity and the maximum principle bound the circulation $r u_\theta$ on a ball of fixed radius, and the rescaling multiplies it by a factor which vanishes, so that the swirl source of the potential vorticity equation disappears in the limit. The scalar carried into the limit is then either the potential vorticity or the azimuthal vorticity, depending on whether the rescaled axis stays at finite distance or recedes to infinity, and in both cases we prove that the limiting scalar vanishes at the upper endpoint of the rescaled time interval. The limiting rescaled meridional field is therefore bounded, with vanishing azimuthal vorticity and vanishing divergence; hence its gradient vanishes, which contradicts the selection of the points at which we zoomed in. The elimination of the circulation first, and of the potential vorticity second, is also the mechanism of the Euler rigidity theorem in~\cite[Theorem~4.5]{CIV26}; what~\cite{CIVEulerLength} adds is that this argument survives the passage to an ancient limit, in which the transporting field is only a weak-$*$ limit.

The \emph{anisotropy} inherent in the lengths~\eqref{eq:aniso:lengths} forces the proof in this paper to depart from~\cite{CIVEulerLength}. To see this, note that after the zoom-in, the radial diffusion coefficient is $(-t)/\ell_r(t)^2=1$, the axial one is $(-t)/\ell_z(t)^2=(-t)^{2h}$, and the rescaling multiplies the circulation by $(-t)/(\ell_r(t)\ell_z(t))=(-t)^h$. In the isotropic case all three of these quantities are of size $(-t)/\ell(t)^2 \to 0$ as $t\uparrow 0$, and thus all three terms vanish in the zoom-in limit. Here the last two vanish, but the first one does not. The first difficulty is therefore that radial diffusion survives in the limit, which is no longer governed by a transport equation: the step which eliminates the potential vorticity in~\cite{CIV26, CIVEulerLength}, its conservation along backward meridional trajectories, is unavailable. In place of an argument based on backward characteristics, we use Lemma~\ref{lem:aniso:comparison}, which is a comparison principle for bounded distributional solutions of a drift-diffusion equation whose diffusion acts in a (strict) subset of the variables.

As in~\cite{CIVEulerLength}, the zoom-in limit need not be a solution of~\eqref{eq:nse:forced}, so we take the curl of~\eqref{eq:nse:forced} first, which removes the pressure, and carry only the vorticity equations~\eqref{eq:aniso:theta} and~\eqref{eq:aniso:q} into the limit. Whether the rescaled axis stays at finite distance or recedes, the limiting scalar is bounded, solves a homogeneous equation on the whole space, and tends uniformly to zero in the remote past; the comparison principle then implies that it vanishes. Incompressibility and the boundedness of the limiting velocity then give smallness of $(-t)\p_ru_r$ and $(-t)\p_zu_z$, and, when the rescaled axis stays at finite distance, of $(-t)u_r/r$; when it recedes, that quotient is small by~\eqref{eq:aniso:bounds} alone. Lastly, the bound $(-t)|\p_zu_r|\le C(-t)^h$ is immediate from~\eqref{eq:aniso:bounds}, and it gives the relative smallness of this term as well.

The second difficulty is that $\p_ru_z$ may be as large as $C(-t)^{-1-h}$, so that an $o((-t)^{-1})$ bound on the full gradient of the meridional velocity $b = u_r e_r+u_z e_z$ is unavailable at any stage of the argument. We therefore track only the four quantities in 
\begin{align}
G_\rho(t):=\bigl\| 
|\p_r u_r (\cdot,t)|+|u_r(\cdot,t)/r |+|\p_z u_z (\cdot,t)|+|\p_z u_r (\cdot,t)|
\bigr\|_{L^\infty(B(\rho))}
\,,
\label{eq:aniso:G}
\end{align}
where $u_r/r$ is extended continuously across the axis. The second quantity is dominated by the first in the supremum over $B(\rho)$, since $u_r$ vanishes on the axis and the radial segment to the axis stays in $B(\rho)$; we keep it because Section~\ref{sec:aniso:closure} uses $u_r/r$ as the coefficient of $u_\theta$ in~\eqref{eq:aniso:scalar:nse:swirl} and of $\omega_\theta^2$ in~\eqref{eq:aniso:closure:stretch}. The relative smallness result, which replaces~\cite[Proposition~3.1]{CIVEulerLength}, is as follows: 

\begin{proposition}[{\bf Smallness of the meridional quantities}]
\label{prop:aniso:small}
Under the hypotheses of Theorem~\ref{thm:aniso:main}, for every fixed $0<\rho<1$ we have
\begin{align}
\lim_{t\uparrow0}(-t)G_\rho(t)=0
\,.
\label{eq:aniso:small}
\end{align}
\end{proposition}

To deduce regularity at $(0,0)$ from~\eqref{eq:aniso:small}, we first show that for every $\eta>0$ the bound $\|u_\theta(\cdot,t)\|_{L^\infty}\le C_\eta(-t)^{-\eta}$ holds on an interior ball, at all sufficiently late times; we obtain this bound by applying the maximum principle in the swirl equation to the function $(-t)^\eta u_\theta$. We then consider the local enstrophy balance, which contains the nonlinear term $(\p_ru_z)\omega_r\omega_z$, whose coefficient $\p_ru_z$ may be of size $(-t)^{-1-h}$, too large to be placed in $L^\infty$ and fed to Gr\"onwall's inequality. Appealing instead to the identity $\omega_r=-\p_zu_\theta$, we integrate by parts in $z$, which moves the derivative from the swirl onto the product $(\p_ru_z)\omega_z$. As a consequence, the swirl enters only through its supremum, and the two factors of the product are controlled by the enstrophy and the enstrophy dissipation through local div-curl estimates. The resulting enstrophy estimate then closes and gives a Ladyzhenskaya-Prodi-Serrin bound up to time zero, as we show in Section~\ref{sec:aniso:closure}.

\subsection{Prior results on axisymmetric Navier-Stokes}
\label{sec:aniso:prior}
Without swirl, axisymmetric solutions of 3D Navier-Stokes are globally regular; see~Ladyzhenskaya~\cite{Ladyzhenskaya68} and also Ukhovskii and Yudovich~\cite{UY68}. With swirl, the global regularity question is open and only conditional results are available. Singularities which obey a scale-invariant (so-called Type~I) pointwise bound are excluded: $|u|\le C(r^2-t)^{-1/2}$ by Chen, Strain, Yau, and Tsai~\cite{CSTY08}; $r|u|\le C$ by Koch, Nadirashvili, Seregin, and \v{S}ver\'ak~\cite{KNSS09}, and $|u|\le C(-t)^{-1/2}$ by Chen, Strain, Tsai, and Yau~\cite{CSTY09}. Seregin and \v{S}ver\'ak~\cite{SereginSverak09} proved a local Type~I regularity criterion which imposes the bound $\le C (-t)^{-1/2}$ only on the meridional part of the velocity. Zhang~\cite{QiZhang26} recently proved regularity under the one-component one-sided Type~I assumption: $u_r \geq - C (-t)^{-1/2}$ for some $C>0$, with $r u_\theta$ bounded. For a detailed account of the many other conditional results available for axisymmetric Navier-Stokes, we refer to the book of Tsai~\cite[Chapter~10]{Tsai18}.

Past the Type~I threshold, singularities are constrained by the growth of critical norms: the $L^\infty_tL^3_x$ theorem of Escauriaza, Seregin, and \v{S}ver\'ak~\cite{ESS03} was made quantitative by Tao~\cite{Tao21}, and was adapted to the axisymmetric class, with quantitative improvements, by Palasek~\cite{Palasek21} and by O\.za\'nski and Palasek~\cite{OzanskiPalasek23}.

Closer to the hypotheses of this paper, Seregin~\cite{Seregin20} showed that an axisymmetric singularity must be of Type~II, in the sense of local scale-invariant integral quantities. Moreover, in~\cite{Seregin23,Seregin24,Seregin25,Seregin26} he reduced the Type II blowup scenario to a Liouville problem for ancient Euler solutions, which may be ruled out in special cases.

\subsection{Research context for the OpenAI construction}
\label{sec:aniso:builds}
The problem of spontaneous generation of singularities in incompressible fluids has a long history and a strong scientific motivation. Research in this area has produced a vast literature which forms the fertile environment surrounding the OpenAI construction. We do not have space here to do justice to the numerous ideas and results in this area, but it would not be amiss to mention some of the important developments which are
relevant to~\cite{OpenAIManuscript}.

For the incompressible 3D Euler equations, Elgindi~\cite{Elgindi21} proved finite time singularity formation for axisymmetric solutions without swirl, with a velocity of class $C^{1,\alpha}$ for $\alpha>0$ sufficiently small; see also the finite energy version in Elgindi, Ghoul, and Masmoudi~\cite{EGM21}. Earlier this year, blowup constructions in the swirl-free axisymmetric setting were obtained for the entire permissible range $\alpha\in(0,1/3)$, independently, by Shkoller~\cite{Shkoller26},  Chen~\cite{Chen26a,Chen26b}, and Shao, Wei, Zhang, and Zhang~\cite{SWZZ26}. For smooth initial data in the presence of a solid boundary, Chen and Hou~\cite{ChenHouPNAS25} proved finite time blowup for 3D Euler using modulated self-similar analysis and a computer-assisted argument. Closer to the Navier-Stokes equations, Tao~\cite{Tao16} constructed finite time blowup for an ``averaged'' Navier-Stokes equation,  whose modified nonlinearity retains the scaling and energy structure of the original PDE. Numerical and machine-guided searches for a self-similar  blowup profile for 3D Euler and other incompressible fluid models were performed in~\cite{HouHuang22,Hou23,WLGB23,WangEtAl25,GDA26,DGGA26}. These computations have shaped the expectations for a rigorous construction: which scenarios are worth pursuing, and at which regularity a singularity should be looked for.
All results mentioned in this paragraph consider the classical, \emph{unforced}, problem.\footnote{We note that recently OpenAI has also claimed to prove finite time blowup for the 3D Euler equations, from smooth compactly supported initial data~\cite{OpenAIEuler}.}

A different route to blowup for the 3D Euler equations (and related models), \emph{in the presence of a body force}, was opened by C\'ordoba and Mart\'inez-Zoroa. The force injects, at a sequence of times accumulating at the blowup time, small perturbations at ever higher frequencies, which the nonlinearity amplifies. The question is then how smooth and how small the force can be made. We refer to~\cite{CordobaMartinezZoroa23} for the 3D Euler result, with a force which lies in $C^{1,1/2-\eps}\cap L^2$ uniformly in time, and to~\cite{CordobaMartinezZoroa24} for the 2D incompressible porous media result, with a source which is $C^\infty$ smooth in space, uniformly in time. In this direction, Alp\"oge and Buckmaster have reached a force which is $C^\infty$ in both space and time. Very recently they announced finite time blowup for the 3D Euler equations~\cite{AlpogeBuckmasterEuler26}, for the inviscid 2D Boussinesq system~\cite{AlpogeBuckmasterBoussinesq26}, and, with Coiculescu, for the 2D incompressible porous media equation~\cite{AlpogeBuckmasterCoiculescu26}. Each of these three papers considers a $C^\infty_{x,t}$ smooth force.

At the technical level, the OpenAI manuscript names two additional research programs whose ideas have influenced its construction. For the evolution of the wavevectors and of the velocity polarizations of the oscillatory pulses along the background flow,~\cite{OpenAIManuscript} records the methods of Lifschitz and Hameiri~\cite{LifschitzHameiri91} and of Friedlander and Vishik~\cite{FriedlanderVishik91}, which detect the instability of an inviscid flow under short wavelength perturbations. For the realization of a prescribed stress by the addition of high frequency oscillations,~\cite{OpenAIManuscript} 
records the use of ideas from the convex integration program for incompressible fluid equations, introduced by De~Lellis and Sz\'ekelyhidi~\cite{DeLellisSzekelyhidi09,DeLellisSzekelyhidi13}, and then further developed in~\cite{DaneriSzekelyhidi17,Isett18,BuckmasterVicol19} and many others.  

We emphasize that several decades of conditional regularity results restrict the form that a putative 3D Navier-Stokes singularity may take; this helps tremendously if one tries to narrow the search space for a blowing up solution. For a suitable weak solution whose force is bounded, the partial regularity theory of Caffarelli, Kohn, and Nirenberg~\cite{CKN} forces the singular set to have vanishing one-dimensional parabolic Hausdorff measure. The singularity, moreover, cannot be of the backward self-similar form proposed by Leray, cf.~Ne\v{c}as, R\r{u}\v{z}i\v{c}ka, and \v{S}ver\'ak~\cite{NRS96} for $L^3$ profiles, Tsai~\cite{Tsai98} for $L^p$ profiles with $3<p<\infty$ under local energy assumptions, and Pineau and Vicol~\cite{PineauVicol26} for rotated self-similar profiles which obey a Type I bound, at small or large rotation speeds. Nor can the direction of the vorticity stay coherent, by the alignment criterion of Constantin and Fefferman~\cite{ConstantinFefferman93}. In the axisymmetric class, the Type~I bounds recalled in Section~\ref{sec:aniso:prior} likewise exclude singularities. Self-similar Euler profiles with similarity exponent below $1/2$, which is the range that a Navier-Stokes zoom at an Euler length would require, are also excluded when the profile is smooth and satisfies an outgoing property~\cite{CIV26}. Note that apart from the partial regularity theory, these results consider the unforced Navier-Stokes problem. The above list is far from being complete; we refer to the book of Lemari\'e-Rieusset~\cite{LemarieRieusset16} for an excellent recent survey. 

\begin{remark}[{\bf Forced singularity constructions for shell models}]
\label{rem:shell:models}
Returning to the forced constructions, we draw an analogy with the finite time singularities constructed for model equations motivated by Navier-Stokes. Palasek~\cite[Theorem~1.3 and Remark~1.5]{Palasek26} recently proved finite time blowup in the energy supercritical range for a viscous shell model of Obukhov type, with super-exponentially separated frequencies. His initial data is smooth, and the smooth force he considers is such that its norms all tend to zero at the blowup time~\cite[Remark~1.4]{Palasek26}. The force is nonzero at every time before the blowup, so it vanishes identically on no time interval ending at the singularity. The force is a residual, as in~\cite{OpenAIManuscript}: here it is the dissipation of each mode, switched off until that mode is activated~\cite[Section~3.3]{Palasek26}. Interestingly, Palasek notes that the force cannot be dropped in the parameter range he associates to three dimensions~\cite[Remark~1.4 and Section~1.3.2]{Palasek26}; this statement is offered without proof, and is attributed to forthcoming work of Looi.
\end{remark}

\subsection{Organization of the paper}
\label{sec:aniso:organization}

Section~\ref{sec:reduction} proves the analytic symmetry reduction, deduces Theorem~\ref{thm:main} from Theorem~\ref{thm:aniso:main}, and records the consequences for the force. Sections~\ref{sec:aniso:identities}--\ref{sec:aniso:closure} prove Theorem~\ref{thm:aniso:main}. Appendix~\ref{app:scales} records the properties claimed for the OpenAI construction which enter the hypotheses of our results, and what those results give for it.

\section{Reduction to the axisymmetric problem}
\label{sec:reduction}

In this section we deduce Theorem~\ref{thm:main} from Theorem~\ref{thm:aniso:main}. Throughout this section $v=\mathcal Pu$ and $w=u-v$ are as in~\eqref{eq:interior:average}, $\mathcal Q=B(1)\times(-1,0)$, and $(u,\pi)$ is smooth on compact subsets of the cylinder on which it is given, so that $f$ is smooth there as well.

\subsection{Spatial analyticity at a fixed time before blowup}
\label{sec:exterior:analytic}
That $u(\cdot,t)$ is real analytic in space at each fixed $t<0$ is classical. We recall the statement in the form in which it is used here and in Remark~\ref{rem:exterior:nonanalytic}. No analyticity in time is assumed.

\begin{theorem}[{\bf Interior spatial analyticity}~{\cite[Theorem~1.2]{Kahane69}}]
\label{thm:analytic:interior}
Let $R>0$, $t_1<t_2$, and let $(u,\pi)$ be a solution of~\eqref{eq:nse:forced} on the cylinder $B(R)\times(t_1,t_2)$, smooth on compact subsets of that cylinder. Assume that the force is real analytic in the space variables, locally uniformly on cylinders compactly contained in $B(R)\times(t_1,t_2)$: for every $R'<R$ and every compact interval $J\subset(t_1,t_2)$ there are $M<\infty$ and $a>0$ such that\footnote{Since $k!\le k^k\le e^kk!$ for $k\ge1$, the hypothesis~\eqref{eq:analytic:interior:force} is unchanged if $|\alpha|!$ is replaced by $|\alpha|^{|\alpha|}$, which is the form in which Kahane states it on p.~387 of~\cite{Kahane69}.}
\begin{align}
\sup_{t\in J}\norm{\p_x^\alpha f(\cdot,t)}_{L^\infty(B(R'))}
\le Ma^{-|\alpha|}|\alpha|!
\qquad\mbox{for all }\alpha\in\mathbb N_0^3
\,.
\label{eq:analytic:interior:force}
\end{align}
Then, the velocity obeys a bound of the same form: for every $R'<R$ and every compact interval $J\subset(t_1,t_2)$ there are $M'<\infty$ and $a'>0$ such that
\begin{align}
\sup_{t\in J}\norm{\p_x^\alpha u(\cdot,t)}_{L^\infty(B(R'))}
\le M'(a')^{-|\alpha|}|\alpha|!
\qquad\mbox{for all }\alpha\in\mathbb N_0^3
\,.
\label{eq:analytic:interior:velocity}
\end{align}
In particular, $u(\cdot,t)$ is real analytic on $B(R)$ for every $t\in(t_1,t_2)$. The constants $M'$ and $a'$ depend on $R'$ and $J$, and they need not be uniform up to the parabolic boundary: $a'$ may tend to zero as $J$ approaches an endpoint of $(t_1,t_2)$.
\end{theorem}

For a conservative force this is Kahane's theorem~\cite[Theorem~1.2]{Kahane69}, proved there for weak solutions, under the hypotheses of Serrin's interior regularity theorem~\cite[Section~3]{Serrin62}. For a nonconservative force, Kahane states on p.~387 of~\cite{Kahane69} that the conclusion of~\cite[Theorem~1.2]{Kahane69} persists once $f$ obeys~\eqref{eq:analytic:interior:force} in its $|\alpha|^{|\alpha|}$ form, with analyticity understood as the bound~\cite[equation~(1.5)]{Kahane69}, which is~\eqref{eq:analytic:interior:velocity} above. The point is not taken up again in that paper, so we indicate why the proof goes through. First, the spatial derivatives are estimated in~\cite{Kahane69} through the representation formulas of Serrin~\cite[equations~(18) and~(20)]{Serrin62}. With a force, the heat potential for the vorticity carries $f$ at order zero, in the slot of the bilinear term~\cite[Section~4, equation~(23) and the display which follows it]{Serrin62}. Second, adjoining $f$ to the unknown $(u,\omega)$ of the abstract integral system of~\cite[Theorem~1.3]{Kahane69} makes the force one more component of the inhomogeneity, in rows whose two kernel blocks vanish. Third, the two hypotheses of that theorem hold: the unknown, now the triple $(u,\omega,f)$, is bounded on every compact cylinder, by smoothness here and by~\cite[Theorem~1.1]{Kahane69} in~\cite{Kahane69}; and the inhomogeneity obeys~\cite[equation~(1.5)]{Kahane69}, which for $f$ is the bound~\eqref{eq:analytic:interior:force}. The inductive estimates of~\cite[Section~4]{Kahane69} then apply without change.

The hypothesis~\eqref{eq:interior:force:analytic} of Theorem~\ref{thm:main} is~\eqref{eq:analytic:interior:force} on the unit parabolic cylinder. In this paper we apply Theorem~\ref{thm:analytic:interior} in the following form.

\begin{corollary}[{\bf Spatial analyticity at a fixed time before blowup}]
\label{lem:analytic:slice}
Let $R,\delta>0$, and let $(u,\pi)$ be a solution of~\eqref{eq:nse:forced} on $B(R)\times(-\delta,0)$, smooth on compact subsets of that cylinder. Suppose that for some $0<R'\le R$, some $-\delta<t_1<t_2<0$, and some $M<\infty$ and $a>0$,
\begin{align}
\sup_{t\in[t_1,t_2]}\norm{\p_x^\alpha f(\cdot,t)}_{L^\infty(B(R'))}
\le Ma^{-|\alpha|}|\alpha|!
\qquad\mbox{for all }\alpha\in\mathbb N_0^3
\,.
\label{eq:analytic:slice:force}
\end{align}
Then, for every $t\in(t_1,t_2)$, the velocity $u(\cdot,t)$ is real analytic on $B(R')$.
\end{corollary}

This is Theorem~\ref{thm:analytic:interior} on the cylinder $B(R')\times(t_1,t_2)$, on which~\eqref{eq:analytic:slice:force} gives~\eqref{eq:analytic:interior:force} with the same $M$ and $a$. Since the zero force satisfies~\eqref{eq:analytic:slice:force} on every cylinder, for the unforced equations $u(\cdot,t)$ is real analytic on $B(R)$ at every $t\in(-\delta,0)$.

We note that analyticity in space and time for the Navier-Stokes equations was proved by Masuda~\cite{Masuda67}; this result was extended to uniform analyticity up to an analytic boundary by Komatsu~\cite{Komatsu80}. Quantitative lower bounds on the radius of spatial analyticity are due to Gruji\'c and Kukavica~\cite{GK98}, in terms of $L^p$ norms of the initial datum. At interior points, and for a divergence-free force analytic on a ball, this result is due to Bradshaw, Gruji\'c, and Kukavica~\cite{BGK15}.

\subsection{Axisymmetry on a fixed ball, and the proof of Theorem~\texorpdfstring{\ref{thm:main}}{1.1}}
\label{sec:reduction:core}

\begin{proof}[Proof of Theorem~\ref{thm:main}]
By Corollary~\ref{lem:analytic:slice} on $\mathcal Q$, applied to each cylinder $B(R)\times J$ in~\eqref{eq:interior:force:analytic}, the velocity $u(\cdot,t)$ is real analytic on $B(1)$ for every $t\in(-1,0)$. We fix such a $t$ and an angle $\varphi\in[0,2\pi)$. With the notation in~\eqref{eq:interior:average}, we define
\begin{align}
F_\varphi(x)=u(x,t)-(\mathscr R_\varphi u)(x,t)=u(x,t)-Q_\varphi u(Q_\varphi^{-1}x,t)
\,.
\label{eq:interior:defect}
\end{align}
Since $Q_\varphi$ preserves $B(1)$, the field $F_\varphi$ is real analytic on $B(1)$. Since $Q_\varphi$ also preserves the core ball $B(\rho(t))$, on which $u=v$ by~\eqref{eq:interior:core}, and since $v=\mathcal Pu$ is invariant under $\mathscr R_\varphi$, we have $F_\varphi=v-\mathscr R_\varphi v=0$ on $B(\rho(t))$. The identity theorem on the connected set $B(1)$ then gives $F_\varphi\equiv0$ on $B(1)$. This holds for every $\varphi\in[0,2\pi)$ and every $t\in(-1,0)$. Averaging the identity $u=\mathscr R_\varphi u$ over $\varphi$, and recalling the definition~\eqref{eq:interior:average} of $\mathcal P$, we obtain $u=\mathcal Pu$, that is,
\begin{align}
u=v\,,\qquad w=0
\quad\mbox{on }\mathcal Q
\,.
\label{eq:interior:global:symmetry}
\end{align}
Note that this argument is made at each fixed time, so that the core radius $\rho(t)$ and the analyticity radii may tend to zero as $t\uparrow0$, at unrelated rates.

Next, we average the Navier-Stokes momentum equation~\eqref{eq:nse:forced} with respect to the azimuthal angle. Since, by~\eqref{eq:interior:global:symmetry}, $u$ is fixed by every $\mathscr R_\varphi$, and since $\mathscr R_\varphi$ commutes with $\p_t$, with $\Delta$, and with $u\mapsto(u\cdot\nabla)u$, applying $\mathcal P$ to~\eqref{eq:nse:forced} leaves the three velocity terms unchanged and replaces only the pressure and the force by their angular means
\begin{align*}
\bar\pi(x,t)=\frac1{2\pi}\int_0^{2\pi}\pi(Q_\varphi^{-1}x,t)\,d\varphi
\,,\qquad
\bar f=\mathcal Pf
\,.
\end{align*}
The averaged pressure remains in $L^{3/2}_{x,t}$ (see~\eqref{eq:interior:energy:class}), and~\eqref{eq:interior:force:c-two} gives $\|\bar f\|_{L^\infty_tC_x^2}\le CM_f$. The pair $(u,\bar\pi)$ is suitable with the force $\bar f$: it is smooth on compact subsets of $\mathcal Q$ and solves the averaged equation classically, so the local energy inequality is an identity, with $\bar\pi\in L^{3/2}_{x,t}(\mathcal Q)$ and $\bar f$ bounded.

Thus $(u,\bar\pi,\bar f)$ is an axisymmetric suitable weak solution of~\eqref{eq:nse:forced} on $\mathcal Q$, in the class of Section~\ref{sec:aniso:axisymmetric}, whose force obeys~\eqref{eq:interior:force:c-two} with $CM_f$ in place of $M_f$. Moreover, by~\eqref{eq:interior:global:symmetry} the bounds~\eqref{eq:interior:mean:bounds} on the angular mean are the bounds~\eqref{eq:aniso:bounds} on $u$ itself, with the same exponent $h$ and the same constant $\Caxi$. Theorem~\ref{thm:aniso:main} makes $(0,0)$ a regular point, which is~\eqref{eq:interior:regular}.
\end{proof}

\subsection{Consequences for the force}
\label{sec:reduction:force}
The analyticity of the force enters the proof of Theorem~\ref{thm:main} at exactly one point, the passage from symmetry on the shrinking core~\eqref{eq:interior:core}, to the symmetry~\eqref{eq:interior:global:symmetry} on the fixed ball; for a merely smooth force we do not know how to make that passage. At a singular point, we may restate Theorem~\ref{thm:main} as a statement about the force alone.

\begin{corollary}[{\bf The force under an axisymmetric core}]
\label{cor:interior:nonanalytic}
Let $(u,\pi)$ be a suitable weak solution of~\eqref{eq:nse:forced} on $\mathcal Q$, smooth on compact subsets of $\mathcal Q$, with a force $f$ satisfying~\eqref{eq:interior:force:c-two},  assume~\eqref{eq:interior:mean:bounds} for some $0<h<1/2$  and $\Caxi<\infty$, and assume~\eqref{eq:interior:core}. If $(0,0)$ is a singular point, then $f$ violates~\eqref{eq:interior:force:analytic}, and in particular $f\not\equiv0$ there. Moreover, $f$ then does not vanish identically on any cylinder $B(R')\times(-\delta',0)$ with $0<R'\le1$ and $0<\delta'\le1$.
\end{corollary}

\begin{proof}[Proof of Corollary~\ref{cor:interior:nonanalytic}]
The first assertion is the contrapositive of Theorem~\ref{thm:main}, and $f=0$ satisfies~\eqref{eq:interior:force:analytic}. For the second, were $f$ to vanish identically on such a cylinder, the hypotheses of Theorem~\ref{thm:main} would hold there with $\rho(t)$ replaced by $\min\{\rho(t),R'/2\}$, and the Navier-Stokes scaling would make $(0,0)$ regular.
\end{proof}

\begin{remark}[{\bf Upper bounds in place of exact vanishing}]
\label{rem:interior:approximate}
We note that exact axisymmetry on the core may be replaced by a bound on the non-axisymmetric part $w$ of the velocity, on a fixed ball. That is, if in Theorem~\ref{thm:main} we drop~\eqref{eq:interior:core} and~\eqref{eq:interior:force:analytic} and assume instead that 
\begin{align}
M_w:=\sup_{-1<t<0}\norm{w(\cdot,t)}_{C_x^3(B(1))}<\infty
\,,
\label{eq:interior:error:bound}
\end{align}
then $(0,0)$ is still regular. No smallness of $M_w$ is required, and analyticity plays no role. Indeed, averaging~\eqref{eq:nse:forced} over rotations kills the mixed terms $(v\cdot\nabla)w$ and $(w\cdot\nabla)v$, since $v$ is invariant and $\mathcal Pw=0$, and gives
\begin{subequations}
\label{eq:interior:mean:forced}
\begin{align}
\p_tv+(v\cdot\nabla)v-\Delta v+\nabla\bar\pi
&=F+\bar f\,,\\
\div v&=0\,,\\
F&=-\mathcal P\bigl((w\cdot\nabla)w\bigr)\,,
\\
\sup_{-1<t<0}\norm{F(\cdot,t)}_{C_x^2(B(1))}
&\le C M_w^2
\,,
\end{align}
\end{subequations}
where $C$ is a universal constant, and $\bar f=\mathcal Pf$. Averaging does not increase the norms in~\eqref{eq:interior:energy:class}, and $(v,\bar\pi)$ is smooth on compact subsets of $\mathcal Q$, so it is a suitable weak solution of~\eqref{eq:interior:mean:forced} with the force $F+\bar f$, whose $C^2$ norm is at most $CM_w^2+CM_f$. Theorem~\ref{thm:aniso:main} makes $v$ bounded near $(0,0)$, and~\eqref{eq:interior:error:bound} makes $u=v+w$ bounded there as well. We emphasize that the identity theorem uses the exact vanishing in~\eqref{eq:interior:core}, and that we do not know how to replace it by an algebraic rate of smallness.  Section~\ref{sec:properties:criteria} shows that~\eqref{eq:interior:error:bound} fails for the OpenAI construction.
\end{remark}

\begin{remark}[{\bf The symmetry of the curl of the force}]
\label{rem:force:symmetry}
No symmetry of $f$ is assumed in Theorem~\ref{thm:main}, but the hypotheses~\eqref{eq:interior:core} and~\eqref{eq:interior:force:analytic} together imply one: for every $\varphi\in[0,2\pi)$,
\begin{align}
\mathscr R_\varphi(\curl f)=\curl f\,,
\qquad\mbox{on }\mathcal Q
\,.
\label{eq:force:symmetry}
\end{align}
Indeed, once $\mathscr R_\varphi u=u$ is known from the proof of Theorem~\ref{thm:main}, applying $\mathscr R_\varphi$ to the momentum equation gives $f-\mathscr R_\varphi f=\nabla(\pi-\pi\circ Q_\varphi^{-1})$, and~\eqref{eq:force:symmetry} follows upon taking the curl. It would be interesting to decide whether in Theorem~\ref{thm:main} assumption~\eqref{eq:interior:force:analytic} may be replaced by the rotation invariance~\eqref{eq:force:symmetry} of $\curl f$.
\end{remark}

\begin{remark}[{\bf Exterior vanishing and analyticity}]
\label{rem:exterior:nonanalytic}
A second obstruction to the analyticity of the force uses neither~\eqref{eq:interior:mean:bounds} nor~\eqref{eq:interior:core}, only an open set off the axis on which the meridional velocity vanishes. Let $(u,\pi)$ be a solution of~\eqref{eq:nse:forced} on $B(R)\times(-\delta,0)$, smooth on compact subsets of that cylinder, and let $E\subset B(R)\cap\{r>0\}$ be a nonempty open set such that 
\begin{align}
u_r=u_z=0
\qquad\mbox{on }E\times(-\delta,0)
\,.
\label{eq:exterior:vanishing}
\end{align}
Suppose that $u_z(0,t_0)\ne0$ for some $t_0\in(-\delta,0)$.\footnote{For the OpenAI construction this holds at every time sufficiently close to the blowup time, where $u_z(0,t)=j_0(-t)^{-1/2-h}\ne0$ by~\eqref{eq:properties:axis}. The remark thus says that a flow which is pure swirl on an open set off the axis, and has nonzero axial velocity at a point of the axis, cannot arise from a spatially analytic force: analyticity would carry the exterior vanishing of $u_z$ to the axis.} Then, on every ball $B(R')$ with $0<R'\le R$ which meets $E$, and for every $t_1\in(-\delta,t_0)$, there are no constants $M<\infty$ and $a>0$ such that
\begin{align}
\sup_{t\in[t_1,t_0]}\norm{\p_x^\alpha f(\cdot,t)}_{L^\infty(B(R'))}
\le Ma^{-|\alpha|}|\alpha|!
\qquad\mbox{for all }\alpha\in\mathbb N_0^3
\,.
\label{eq:exterior:force:analytic:local}
\end{align}
In particular, $f$ does not vanish identically on $B(R')\times(t_1,t_0)$, since the zero force satisfies~\eqref{eq:exterior:force:analytic:local}.
The proof is short. Suppose that~\eqref{eq:exterior:force:analytic:local} holds for some $M$ and $a$. By Corollary~\ref{lem:analytic:slice}, $u(\cdot,t)$ is then real analytic on $B(R')$ for every $t\in(t_1,t_0)$. At each such $t$, the scalar $u_z(\cdot,t)$ vanishes on the nonempty open set $E\cap B(R')$ by~\eqref{eq:exterior:vanishing}, hence on all of $B(R')$ by the identity theorem, and in particular $u_z(0,t)=0$. Letting $t\uparrow t_0$ and using the continuity of $u_z(0,\cdot)$, we obtain $u_z(0,t_0)=0$, a contradiction.
\end{remark}

\section{Axisymmetric identities and auxiliary results}
\label{sec:aniso:identities}

We now turn to the proof of Theorem~\ref{thm:aniso:main}. Throughout Sections~\ref{sec:aniso:identities}--\ref{sec:aniso:closure}, $(u,\pi)$ is an axisymmetric suitable weak solution of~\eqref{eq:nse:forced} on $\mathcal Q$, in the class of Section~\ref{sec:aniso:axisymmetric}, with a force satisfying~\eqref{eq:interior:force:c-two}; averaging over rotations as in Section~\ref{sec:reduction:core}, we may (and do) take $\pi$ and $f$ axisymmetric as well. This section records a few of the auxiliary results that the proof needs: the equations in cylindrical variables, a maximum principle for functions vanishing on the axis (Lemma~\ref{lem:aniso:axis}), a regular annulus with a bounded circulation (Lemma~\ref{lem:aniso:annulus}), and a comparison principle for the bounded scalars produced by the ancient limits of Section~\ref{sec:aniso:zoom} (Lemma~\ref{lem:aniso:comparison}). Apart from the force, the identities and Lemma~\ref{lem:aniso:annulus} are the same as those of the companion paper~\cite{CIVEulerLength}; the force enters as $(\curl f)_\theta$ in~\eqref{eq:aniso:theta}, as $g=(\curl f)_\theta/r$ in~\eqref{eq:aniso:q}, and as $rf_\theta$ in~\eqref{eq:aniso:circulation:pde}, and~\eqref{eq:interior:force:c-two} bounds all three. We denote by $C$ a constant which may depend on $h$, $M_f$, the constant $\Caxi$ of~\eqref{eq:aniso:bounds}, and fixed interior radii, but never on a sequence of rescaling times.

\subsection{Axisymmetric identities}
\label{sec:aniso:equations}

In the cylindrical frame $(e_r,e_\theta,e_z)$, and with $\pi$ and $f$ axisymmetric, the system~\eqref{eq:nse:forced} reads as
\begin{subequations}
\label{eq:aniso:scalar:nse}
\begin{align}
\p_tu_r+u_r\p_ru_r+u_z\p_zu_r-\frac{u_\theta^2}{r}
&=\left(\p_r^2+\frac1r\p_r+\p_z^2-\frac1{r^2}\right)u_r-\p_r\pi+f_r\,,
\label{eq:aniso:scalar:nse:r}
\\
\p_tu_z+u_r\p_ru_z+u_z\p_zu_z
&=\left(\p_r^2+\frac1r\p_r+\p_z^2\right)u_z-\p_z\pi+f_z\,,
\label{eq:aniso:scalar:nse:z}
\\
\p_tu_\theta+u_r\p_ru_\theta+u_z\p_zu_\theta+\frac{u_r}{r}u_\theta
&=\left(\p_r^2+\frac1r\p_r+\p_z^2-\frac1{r^2}\right)u_\theta+f_\theta\,,
\label{eq:aniso:scalar:nse:swirl}
\\
\p_ru_r+\frac{u_r}{r}+\p_zu_z&=0
\,.
\label{eq:aniso:scalar:nse:div}
\end{align}
\end{subequations}
We call $b=u_re_r+u_ze_z$ the meridional velocity and $u_\theta e_\theta$ the swirl, and we write the vorticity components, the circulation, and the potential vorticity as
\begin{align}
\omega_r=-\p_z u_\theta\,,\qquad
\omega_\theta=\p_z u_r-\p_r u_z\,,\qquad
\omega_z=\p_r u_\theta+\frac{u_\theta}{r}\,,\qquad
\Gamma=ru_\theta\,,\qquad \Omega=\frac{\omega_\theta}{r}
\,.
\label{eq:aniso:vorticity}
\end{align}
Applying $\p_z$ to~\eqref{eq:aniso:scalar:nse:r} and $\p_r$ to~\eqref{eq:aniso:scalar:nse:z} and subtracting, we eliminate the pressure and obtain
\begin{align}
\p_t\omega_\theta+u_r\p_r\omega_\theta+u_z\p_z\omega_\theta
&=\frac{u_r}{r}\omega_\theta+\frac1r\p_z(u_\theta^2)
+\left(\p_r^2+\frac1r\p_r+\p_z^2-\frac1{r^2}\right)\omega_\theta
+(\curl f)_\theta
\,,
\label{eq:aniso:theta}
\end{align}
where $(\curl f)_\theta=\p_zf_r-\p_rf_z$. Dividing~\eqref{eq:aniso:theta} by $r$, and using the identity $r^{-1}(\p_r^2+r^{-1}\p_r-r^{-2})(r\Omega)=(\p_r^2+3r^{-1}\p_r)\Omega$, we also obtain
\begin{align}
\p_t\Omega+u_r\p_r\Omega+u_z\p_z\Omega
&=\left(\p_r^2+\frac3r\p_r+\p_z^2\right)\Omega
+\p_z\left(\frac{u_\theta^2}{r^2}\right)
+g
\,,\qquad
g:=\frac{(\curl f)_\theta}{r}
\,.
\label{eq:aniso:q}
\end{align}
Lastly, upon multiplying the swirl equation~\eqref{eq:aniso:scalar:nse:swirl} by $r$, we obtain
\begin{align}
\p_t\Gamma+u_r\p_r\Gamma+u_z\p_z\Gamma
=\left(\p_r^2-\frac1r\p_r+\p_z^2\right)\Gamma
+rf_\theta
\,.
\label{eq:aniso:circulation:pde}
\end{align}

Near the axis, a smooth axisymmetric field has components $u_r=r\,a(r^2,z,t)$, $u_\theta=r\,c(r^2,z,t)$, and $u_z=d(r^2,z,t)$, with $a$, $c$, and $d$ smooth; thus $\Omega$, $u_\theta/r$, and $g$ are smooth functions of $(r^2,z,t)$, and~\eqref{eq:aniso:q} holds across the axis once its radial operator $\p_r^2+3r^{-1}\p_r$ is read as the radial Laplacian of $\RR^4$, as in~\cite{KNSS09}. Since $u_r$, $\p_zu_r$, $\p_ru_z$, $\p_zf_r$, and $\p_rf_z$ vanish on the axis, the quotients $u_r/r$, $\p_zu_r/r$, $\p_ru_z/r$, and $g$ are bounded, along the radial segment to the axis, by the suprema of $\p_ru_r$, $\p_r\p_zu_r$, $\p_r^2u_z$, $\p_r\p_zf_r$, and $\p_r^2f_z$ on $B(1)$; in particular
\begin{align}
|g|\le CM_f
\qquad\mbox{on }\mathcal Q
\,.
\label{eq:interior:force:quotient}
\end{align}

\subsection{A maximum principle for functions vanishing on the axis}
\label{sec:aniso:axis}

The swirl equation~\eqref{eq:aniso:scalar:nse:swirl} and the circulation equation~\eqref{eq:aniso:circulation:pde} have coefficients which are singular on the axis, and, unlike~\eqref{eq:aniso:q}, they are not lifted to a higher dimension. The proofs of Lemmas~\ref{lem:aniso:annulus} and~\ref{lem:aniso:closure}, and Remark~\ref{rem:interior:meridional}, apply the maximum principle to them in the following form.

\begin{lemma}[{\bf Maximum principle for functions vanishing on the axis}]
\label{lem:aniso:axis}
Let $R>0$, $t_1<s$, $k\in\RR$, and $M\ge0$, and define the set $D=\{(x,t)\colon x\in B(R),\ r=|x'|>0,\ t_1<t\le s\}$. Let $\phi=\phi(r,z,t)$ be continuous on $\overline{B(R)}\times[t_1,s]$, with $\phi(0,z,t)=0$ for $|z|\le R$ and $t_1\le t\le s$. Assume that $\phi$ is of class $C^2$ in $x$ and $C^1$ in $t$ on $D$, and satisfies
\begin{align}
\p_t\phi+b_r\p_r\phi+b_z\p_z\phi+\gamma\phi
=\Bigl(\p_r^2+\frac kr\p_r+\p_z^2\Bigr)\phi+F
\,,
\label{eq:aniso:axis:pde}
\end{align}
there. Here $b_r$, $b_z$, $\gamma$, and $F$ are real-valued functions on $D$, with $\gamma\ge0$ and $|F|\le M$; no bound is assumed on $b_r$, $b_z$, or $\gamma$, which may be unbounded as $r\to0$. Then
\begin{align}
|\phi(x,t)|\le\max_{\Sigma}|\phi|+M(t-t_1)
\qquad\mbox{on }\overline{B(R)}\times[t_1,s]
\,,
\label{eq:aniso:axis:bound}
\end{align}
where $\Sigma=\bigl(\overline{B(R)}\times\{t_1\}\bigr)\cup\bigl(\partial B(R)\times[t_1,s]\bigr)$ is the parabolic boundary of $B(R)\times(t_1,s]$.
\end{lemma}

The proof of this lemma is elementary; we present it for the convenience of the reader.

\begin{proof}[Proof of Lemma~\ref{lem:aniso:axis}]
Since $-\phi$ satisfies the same assumptions, with $-F$ in place of $F$, it suffices to bound $\phi$ from above. Let $m=\max_{\Sigma}|\phi|\ge0$ and $\eps>0$. Define $\phi_\eps=\phi-(M+\eps)(t-t_1)$, and let $(\bar x,\bar t)$ be a maximum point of the continuous function $\phi_\eps$ on the compact set $K=\overline{B(R)}\times[t_1,s]$. Suppose that $\phi_\eps(\bar x,\bar t)>m$. On $\Sigma$ we have $\phi_\eps\le\phi\le m$, and on the axis we have $\phi_\eps\le\phi=0\le m$. The maximum point $(\bar x,\bar t)$ therefore lies neither on $\Sigma$ nor on the axis, that is, $(\bar x,\bar t)\in D$. At time $\bar t$ the functions $\phi$ and $\phi_\eps$ differ by a constant, so that $(r,z)\mapsto\phi(r,z,\bar t)$ achieves its  maximum on the open half-disc $\{r>0,\ r^2+z^2<R^2\}$ at $\bar x$. Therefore, $\p_r\phi=\p_z\phi=0$, $\p_r^2\phi\le0$, and $\p_z^2\phi\le0$ there. Moreover, $\gamma\phi\ge0$ at $(\bar x,\bar t)$ because $\phi(\bar x,\bar t)\ge\phi_\eps(\bar x,\bar t)>m\ge0$, and $\p_t\phi_\eps\ge0$ there because $\phi_\eps(\bar x,t)\le\phi_\eps(\bar x,\bar t)$ for $t_1\le t<\bar t$ bounds the left derivative from below. At $(\bar x,\bar t)$ the coefficients $b_r$, $b_z$, and $k/r$ are finite, and they multiply first derivatives which vanish. Thus~\eqref{eq:aniso:axis:pde} gives $\p_t\phi_\eps=\p_t\phi-M-\eps\le F-M-\eps\le-\eps$ there, a contradiction. Therefore $\phi\le m+(M+\eps)(t-t_1)$ on $K$, and~\eqref{eq:aniso:axis:bound} follows by letting $\eps\to0$.
\end{proof}

\subsection{The regular annulus}
\label{sec:aniso:regular}
Lemma~\ref{lem:aniso:annulus} below is used twice in the proof of Theorem~\ref{thm:aniso:main}: its circulation bound~\eqref{eq:aniso:circulation:bound} removes the swirl source in the zoom of Section~\ref{sec:aniso:zoom}, and its derivative bounds on the annulus, which hold up to the blowup time, control the cutoff errors of the enstrophy estimate of Section~\ref{sec:aniso:closure}, which are supported in that annulus.

\begin{lemma}[{\bf A regular annulus and bounded circulation}]
\label{lem:aniso:annulus}
Let $0\le R_1<R_0<1$. There are radii $R_1<R_-<R_+<R_0$ and a time $t_0\in(-1,0)$ with the following property. The velocity $u$, together with $\nabla u$ and $\nabla^2u$, is bounded on the cylindrical annulus $\{R_-<|x|<R_+\}\times(t_0,0)$, and, for each $R_*\in(R_-,R_+)$, there is a finite constant $C_\Gamma$ such that 
\begin{align}
\norm{\Gamma}_{L^\infty(B(R_*)\times(t_0,0))}\le C_\Gamma
\,.
\label{eq:aniso:circulation:bound}
\end{align}
\end{lemma}

\begin{proof}[Proof of Lemma~\ref{lem:aniso:annulus}]
Let $S_0\subset\overline{B(R_0)}$ be the set of $x$ for which $(x,0)$ is a singular point. Applied to the pair with the divergence-free force of footnote~\ref{foot:aniso:pressure}, partial regularity gives $\mathcal H^1(S_0)=0$. Indeed, the $\eps$-regularity criterion of Gustafson, Kang, and Tsai~\cite[Theorem~1.1(ii)]{GKT07}, with $(p,q)=(2,2)$ and the force in the parabolic Morrey space $M^{2,2}$ of~\cite[Equation~(6)]{GKT07}, applies at the points $(x,0)$ of the top boundary of $\mathcal Q$.\footnote{In~\cite{GKT07} the criterion is stated for points in the interior of the space-time domain, whereas here it is applied at points of the top boundary, at the blowup time. This is legitimate for two reasons. First, the proof in~\cite{GKT07} uses the solution only on the backward cylinders $B(x,\varrho)\times(-\varrho^2,0)$, which lie inside $\mathcal Q$, and only through the equations and the local energy inequality on them. Second, the local energy inequality at a time $s<0$ involves the test function only on $(-1,s]$, so it holds for test functions which do not vanish at $t=0$; for a pair smooth on compact subsets of $\mathcal Q$ it is in fact an identity at every $s<0$.} The covering argument of~\cite{CKN} (see also~\cite{Lin98,LadyzhenskayaSeregin99,Kukavica08}) is unchanged, since $\int_{B(1)\times(-\delta,0)}|\nabla u|^2\to0$ as $\delta\to0$ by~\eqref{eq:interior:energy:class}. We may therefore choose $R\in(R_1,R_0)$ such that $\partial B(R)$ carries no point of $S_0$, and by compactness there are $\delta>0$, with $R_1<R-\delta<R+\delta<R_0$, and $t_0\in(-1,0)$ such that $u$ is bounded on $\{R-\delta<|x|<R+\delta\}\times(t_0,0)$. We fix any $R_-$ and $R_+$ such that $R-\delta<R_-<R_+<R+\delta$.

Since $u$ is bounded on this cylindrical annulus and $f$ is bounded in $C^2_x$, the interior estimates of Serrin~\cite[Section~4, with $m=1$]{Serrin62} bound $\nabla u$ and $\nabla^2u$ on compact subsets of $\{R-\delta<|x|<R+\delta\}\times(t_0,0)$, by constants which depend only on the subset, on $M_f$, on $\sup|u|$ over the cylindrical annulus, and on $\|\nabla u\|_{L^2(\mathcal Q)}$. These bounds are uniform up to $t=0$, because Serrin's representation formulas~\cite[Equation~(18), Equation~(20), and the display after Equation~(23)]{Serrin62} are one-sided in time: the vorticity at a point $(x,s)$ is the sum of a heat potential of $\omega\wedge u$ and of the force, and of a solution of the heat equation, each bounded at $(x,s)$ in terms of the solution on the backward cylinder $B(x,\varrho)\times(s-\varrho^2,s)$, so that the estimate at $(x,s)$ depends on $s$ only through the norms above. Fixing a small $\varrho$, with $\varrho^2<-t_0$, and replacing $t_0$ by $t_0+\varrho^2$ gives the first assertion.

We fix $R_*\in(R_-,R_+)$. For $s\in(t_0,0)$, the set $\overline{B(R_*)}\times[t_0,s]$ is a compact subset of $\mathcal Q$, so that $u$ is smooth there; the circulation $\Gamma=ru_\theta$ is therefore continuous on this set, smooth off the axis, and, since $|\Gamma|\le r|u|$, it vanishes on the axis. The circulation equation~\eqref{eq:aniso:circulation:pde} is of the form~\eqref{eq:aniso:axis:pde}, with $k=-1$, with no zeroth order term, and with a source which obeys $|rf_\theta|\le M_f$. Lemma~\ref{lem:aniso:axis} therefore bounds $|\Gamma|$ on $\overline{B(R_*)}\times[t_0,s]$ by its values at time $t_0$ and on $\partial B(R_*)\times[t_0,s]$, plus $M_f|t_0|$. The former are bounded since $u$ is smooth at time $t_0$, the latter by $R_*$ times the bound on $u$ on the cylindrical annulus; neither depends on $s$ and letting $s\uparrow0$ gives~\eqref{eq:aniso:circulation:bound}.
\end{proof}

\subsection{A comparison principle with diffusion in fewer variables}
\label{sec:aniso:maximum}

In the ancient limits of Section~\ref{sec:aniso:zoom} the radial diffusion survives the zoom-in procedure, so that the potential vorticity is no longer conserved along backward meridional trajectories, as it is in~\cite{CIV26} and~\cite{CIVEulerLength}. A comparison principle is instead used to make the limiting scalar vanish. That scalar, however, is merely bounded, it solves its equation only in the sense of distributions, and it comes with no bound on its derivatives and no initial trace. Moreover, the diffusion acts in a subset of the variables only (it does not act in the axial variable), so Lemma~\ref{lem:aniso:comparison} is stated in the generality in which it is used.

\begin{lemma}[{\bf Comparison for a bounded distributional solution}]
\label{lem:aniso:comparison}
Let $1\le d\le m$, let $I\subset\RR$ be an open interval, and denote $x=(X,Z)\in\RR^d\times\RR^{m-d}$. Suppose that $B\colon\RR^m\times I\to\RR^m$ is measurable, that $B(\cdot,\tau)\in W^{1,\infty}(\RR^m;\RR^m)$ for almost every $\tau\in I$, and that
\begin{align}
\Lambda_J:=\esssup_{\tau\in J}
\big(\norm{B(\cdot,\tau)}_{L^\infty(\RR^m)}
+\norm{\nabla B(\cdot,\tau)}_{L^\infty(\RR^m)}\big)<\infty
\,,
\label{eq:aniso:comparison:drift}
\end{align}
for every compact interval $J\subset I$. 
Let $q\in L^\infty_{\rm loc}(I;L^\infty(\RR^m))$ satisfy
\begin{align}
\p_\tau q+B\cdot\nabla q=\Delta_Xq
\qquad\mbox{on }\RR^m\times I
\,,
\label{eq:aniso:comparison:equation}
\end{align}
in the sense of distributions.\footnote{Here $B\cdot\nabla q$ really means $\div(Bq)-(\div B)q$.} Then, for almost every $\tau_0\in I$ and almost every $\tau\in I$ with $\tau>\tau_0$,
\begin{align}
\norm{q(\cdot,\tau)}_{L^\infty(\RR^m)}
\le\norm{q(\cdot,\tau_0)}_{L^\infty(\RR^m)}
\,.
\label{eq:aniso:comparison:conclusion}
\end{align}
If in addition $I=(-\infty,T)$ and, for some $C<\infty$ and $\kappa>0$,
\begin{align}
\norm{q(\cdot,\tau)}_{L^\infty(\RR^m)}\le C|\tau|^{-\kappa}
\,,
\label{eq:aniso:ancient:decay}
\end{align}
for almost every $\tau<\min\{-1,T\}$, 
then $q=0$ almost everywhere on $\RR^m\times I$. A continuous representative of $q$ on an open subset vanishes there, and so does any continuous trace at the upper endpoint obtained from that representative.
\end{lemma}

\begin{proof}[Proof of Lemma~\ref{lem:aniso:comparison}]
The argument is standard: we mollify in space, control the commutator as in~\cite{DiPernaLions89}, and compare with a barrier which grows linearly at infinity. 
We fix a compact interval $J=[\tau_-,\tau_+]\subset I$, set $M_\infty=\|q\|_{L^\infty(\RR^m\times J)}$, and mollify in space, $q_\eps=\eta_\eps*q$ with a standard mollifier $\eta_\eps=\eps^{-m}\eta(\cdot/\eps)$ with $\eta$ supported in $B(1)$. Convolving~\eqref{eq:aniso:comparison:equation} with $\eta_\eps$, we obtain
\begin{align}
\p_\tau q_\eps+B\cdot\nabla q_\eps
=\Delta_Xq_\eps+\mathcal R_\eps
\,,\qquad
\mathcal R_\eps=B\cdot\nabla(\eta_\eps*q)-\eta_\eps*(B\cdot\nabla q)
\,.
\label{eq:aniso:comparison:mollified}
\end{align}
Here $\mathcal R_\eps$ is the commutator of DiPerna and Lions, bounded in the usual way: $|\mathcal R_\eps|\le C(m,\eta)\Lambda_JM_\infty$ by~\eqref{eq:aniso:comparison:drift}. Moreover, $\mathcal R_\eps(\cdot,\tau)\to0$ in $L^2_{\rm loc}(\RR^m)$ for almost every $\tau\in J$ by~\cite[Lemma~II.1]{DiPernaLions89}, hence $\mathcal R_\eps\to0$ in $L^2(B(R)\times J)$ for every $R>0$ by dominated convergence. For a fixed $\eps>0$, every term of~\eqref{eq:aniso:comparison:mollified} other than $\p_\tau q_\eps$ is bounded, so $q_\eps$ has a representative which is Lipschitz in time with values in $L^\infty(\RR^m)$; from now on $q_\eps$ denotes this representative. If $\tau_0$ is a Lebesgue point of $\tau\mapsto q(\cdot,\tau)\in L^1_{\rm loc}(\RR^m)$, which is the case for almost every $\tau_0$, then $q_\eps(\cdot,\tau_0)=\eta_\eps*q(\cdot,\tau_0)$ for every $\eps$. We fix such a $\tau_0\in(\tau_-,\tau_+)$ and set $M=\|q(\cdot,\tau_0)\|_{L^\infty(\RR^m)}$, so that $q_\eps(\cdot,\tau_0)\le M$.

With $\langle x\rangle=(1+|x|^2)^{1/2}$ we have $|\nabla\langle x\rangle|\le1$ and $\Delta_X\langle x\rangle\le d$. Upon defining $K=\Lambda_J+d$, the function
\begin{align}
\Psi_\sigma(x,\tau)
=M+\sigma\big(\langle x\rangle+K(\tau-\tau_0)\big)
\,,\qquad\sigma>0
\,,
\label{eq:aniso:comparison:barrier}
\end{align}
satisfies $\p_\tau\Psi_\sigma+B\cdot\nabla\Psi_\sigma-\Delta_X\Psi_\sigma\ge\sigma(K-\Lambda_J-d)=0$ almost everywhere on $\RR^m\times J$, and $\Psi_\sigma\ge M+\sigma|x|$ for $\tau\ge\tau_0$. Hence $v_\eps=(q_\eps-\Psi_\sigma)_+$ vanishes at $\tau=\tau_0$; for $\tau\in[\tau_0,\tau_+]$ it also vanishes on the complement of $B(R_\sigma)$ for $R_\sigma=1+M_\infty/\sigma$. Subtracting the supersolution inequality from~\eqref{eq:aniso:comparison:mollified}, multiplying by $v_\eps$, and integrating over $\RR^m$, we obtain
\begin{align}
\frac12\frac{d}{d\tau}\int v_\eps^2
+\int|\nabla_Xv_\eps|^2
\le\frac12\int(\div B)v_\eps^2
+\int\mathcal R_\eps v_\eps
\,.
\label{eq:aniso:comparison:positive:energy}
\end{align}
Appealing to Gr\"onwall's inequality we obtain $\sup_{\tau\in[\tau_0,\tau_+]}\|v_\eps(\cdot,\tau)\|_{L^2}^2\le C_J\|\mathcal R_\eps\|_{L^2(B(R_\sigma)\times J)}^2\to0$ as $\eps\downarrow0$. Since $q_\eps\to q$ in $L^2_{\rm loc}$, letting $\eps\downarrow0$ and then $\sigma\downarrow0$ gives $q\le M$ almost everywhere on $\RR^m\times(\tau_0,\tau_+)$. The same argument for $-q$, and an exhaustion of $I$ by compact intervals, give~\eqref{eq:aniso:comparison:conclusion}.

For the last assertion, we apply~\eqref{eq:aniso:comparison:conclusion} at times $\tau_j\to-\infty$ at which~\eqref{eq:aniso:ancient:decay} holds, and obtain $\|q(\cdot,\tau)\|_{L^\infty(\RR^m)}\le C|\tau_j|^{-\kappa}$ for almost every $\tau>\tau_j$; thus $q=0$ almost everywhere, and a continuous representative, together with its trace at the upper endpoint, vanishes identically.
\end{proof}

\section{Ancient limits and smallness of the meridional quantities}
\label{sec:aniso:zoom}
The goal of this section is to prove Proposition~\ref{prop:aniso:small}. We argue by contradiction, zooming in where~\eqref{eq:aniso:small} fails, at the radial length $\ell_r$ and the axial length $\ell_z$. Two cases then arise, according to whether the rescaled axis stays at finite distance, or recedes to infinity. Since the pressure is not controlled after the zoom, we pass to the limit in the vorticity equations~\eqref{eq:aniso:theta} and~\eqref{eq:aniso:q}, not in the momentum equation. In both cases the rescaling makes the bounded circulation subcritical, and this removes the swirl source. The radial diffusion, however, survives the rescaling, so that the limiting vorticity vanishes by the comparison principle of Lemma~\ref{lem:aniso:comparison}, not by transport along characteristics. Incompressibility and the boundedness of the limiting meridional velocity then make the selected quantities vanish in the limit, and this contradicts the choice of the points at which we zoom in.

\begin{proof}[Proof of Proposition~\ref{prop:aniso:small}]
Fix an arbitrary $0<\rho<1$. By~\eqref{eq:aniso:bounds} the fourth term in~\eqref{eq:aniso:G} is bounded as $(-t)|\p_zu_r|\le\Caxi(-t)^h$, and thus tends to zero in the limit.

Suppose by contradiction that~\eqref{eq:aniso:small} fails. Then there exist $c_0>0$, times $t_n\uparrow0$, and points $x_n\in B(\rho)$ at which the other three terms satisfy
\begin{align}
(-t_n)\Bigl(|\p_ru_r|+\Bigl|\frac{u_r}{r}\Bigr|+|\p_zu_z|\Bigr)(x_n,t_n)\ge c_0
\,.
\label{eq:aniso:zoom:selected}
\end{align}
By rotation invariance we may without loss of generality take $x_n=(r_n,0,z_n)$, where $r_n = |x_n'|\geq 0$. By Lemma~\ref{lem:aniso:annulus}, applied with $R_1=\rho$ and some $R_0\in(\rho,1)$, there are $\rho<R_*<1$, $t_*\in(-1,0)$, and $C_\Gamma<\infty$ such that
\begin{align}
|ru_\theta(x,t)|\le C_\Gamma
\qquad\mbox{on}\qquad B(R_*)\times(t_*,0)
\,.
\label{eq:aniso:zoom:circulation}
\end{align}
Throughout the proof, $C$ denotes a constant which may depend on $h$, $M_f$, $\Caxi$, $\rho$, $R_*$, $C_\Gamma$, and on the compact time interval under consideration, but never on $n$.

\emph{Step 1: the rescaling.}
We define
\begin{align}
\lambda_n=(-t_n)^{1/2}
\,,\qquad
\delta_n=(-t_n)^h = \lambda_n^{2h}
\,,\qquad
\mu_n=\frac{\lambda_n}{\delta_n}=(-t_n)^{1/2-h}
=\lambda_n^{1-2h}
\,.
\label{eq:aniso:zoom:lengths}
\end{align}
That is, $\lambda_n=\ell_r(t_n)$ and $\mu_n=\ell_z(t_n)$ are the lengths~\eqref{eq:aniso:lengths}, and both vanish as $n\to\infty$ since $0<h<1/2$. Their ratio is $\delta_n=\lambda_n/\mu_n\to0$. Upon passing to a subsequence, we have that $r_n/\lambda_n$ either stays bounded or tends to infinity. Every subsequence extracted in this proof refines the previous ones, and we do not relabel it; since~\eqref{eq:aniso:zoom:selected} holds for every $n$, it holds along every subsequence. We first zoom in about the point $(0,z_n)$ on the axis, in the variables
\begin{align}
R=\frac{r}{\lambda_n}
\,,\qquad
Z = \frac{z-z_n}{\mu_n}
\,,\qquad
\tau = \frac{t}{(-t_n)}
\,,
\label{eq:aniso:zoom:finite:variables}
\end{align}
so that $\tau=-1$ is the selected time $t_n$ and $\tau\le-1$ is its past. We then define
\begin{subequations}
\label{eq:aniso:zoom:fields}
\begin{align}
V_n(R,Z,\tau)&=\lambda_n\,u_r\bigl(\lambda_nR,\,z_n+\mu_nZ,\,(-t_n)\tau\bigr)
\,, \\
W_n(R,Z,\tau)&=\lambda_n\delta_n\,u_z\bigl(\lambda_nR,\,z_n+\mu_nZ,\,(-t_n)\tau\bigr)
\,, \\
S_n(R,Z,\tau)&=\lambda_n\delta_n\,u_\theta\bigl(\lambda_nR,\,z_n+\mu_nZ,\,(-t_n)\tau\bigr)
\label{eq:aniso:zoom:fields:c}
\,, \\
\Omega_n(R,Z,\tau)&=\lambda_n^3\delta_n\,\Omega\bigl(\lambda_nR,\,z_n+\mu_nZ,\,(-t_n)\tau\bigr)
\,,
\end{align}
\end{subequations}
where the cylindrical components on the right are written as functions of $(r,z,t)$. In these variables the bounds~\eqref{eq:aniso:bounds} read as
\begin{subequations}
\label{eq:aniso:zoom:derivatives}
\begin{align}
|\p_R^a\p_Z^bV_n|
&\le \Caxi|\tau|^{-(1+a)/2-(1/2-h)b}
\,,\\
|\p_R^a\p_Z^bW_n|+|\p_R^a\p_Z^bS_n|
&\le \Caxi|\tau|^{-1/2-h-a/2-(1/2-h)b}
\,,
\end{align}
\end{subequations}
for all $a+b\le2$.

These bounds hold at every $(R,Z,\tau)$ for which the point $(\lambda_nR,\,z_n+\mu_nZ,\,(-t_n)\tau)$ lies in $\mathcal Q$, and we call the set of such $(R,Z,\tau)$ the preimage of $\mathcal Q$ under the zoom-in~\eqref{eq:aniso:zoom:finite:variables}. The circulation bound~\eqref{eq:aniso:zoom:circulation} holds on the preimage of the smaller set $B(R_*)\times(t_*,0)$. Since the lengths~\eqref{eq:aniso:zoom:lengths} vanish and $|z_n|<\rho<R_*$, both preimages contain any given compact subset of $[0,\infty)\times\RR\times(-\infty,0)$ for all $n$ large depending on that subset.

Changing variables in the divergence equation~\eqref{eq:aniso:scalar:nse:div}, in the definitions~\eqref{eq:aniso:vorticity}, and in the potential vorticity equation~\eqref{eq:aniso:q}, we obtain the identities
\begin{subequations}
\label{eq:aniso:zoom:finite}
\begin{align}
\p_RV_n+\frac{V_n}{R}+\p_ZW_n&=0
\,,\qquad
\Omega_n=\frac{\delta_n^2\p_ZV_n-\p_RW_n}{R}
\,,
\label{eq:aniso:zoom:finite:identities}\\
\p_\tau \Omega_n+V_n\p_R\Omega_n+W_n\p_Z\Omega_n
&=\left(\p_{RR}+\frac3R\p_R+\delta_n^2\p_{ZZ}\right)\Omega_n
+\p_ZF_n+G_n
\,,\qquad F_n=\frac{S_n^2}{R^2}
\,,
\label{eq:aniso:zoom:finite:equation}
\end{align}
\end{subequations}
at every point of the preimage of $\mathcal Q$, with $R>0$.
Note that the radial diffusion has coefficient one, while the axial diffusion carries the factor $\delta_n^2\to0$.

We record both force terms at once: the force enters~\eqref{eq:aniso:zoom:finite:equation} through $G_n$, and the azimuthal vorticity equation~\eqref{eq:aniso:zoom:receding:equation} of \emph{Step~3} through $H_n$, where
\begin{subequations}
\label{eq:interior:force:scaled}
\begin{align}
G_n(R,Z,\tau)&=\lambda_n^5\delta_n\;g\bigl(\lambda_nR,\,z_n+\mu_nZ,\,(-t_n)\tau\bigr)\,,
\\
H_n(R,Z,\tau)&=\lambda_n^4\delta_n\;(\curl f)_\theta\bigl(r_n+\lambda_nR,\,z_n+\mu_nZ,\,(-t_n)\tau\bigr)
\label{eq:interior:force:scaled:b}
\,.
\end{align}
\end{subequations}
The prefactors are the normalizations $\lambda_n^3\delta_n$ of $\Omega$ in~\eqref{eq:aniso:zoom:fields} and $\lambda_n^2\delta_n$ of $\omega_\theta$ in~\eqref{eq:aniso:zoom:receding:vorticity}, multiplied by the time factor $(-t_n)=\lambda_n^2$. Since $|g|\le CM_f$ on $\mathcal Q$ by~\eqref{eq:interior:force:quotient} and $|(\curl f)_\theta|\le CM_f$ by~\eqref{eq:interior:force:c-two}, both terms tend to zero uniformly: $G_n$ on the preimage of $\mathcal Q$ under~\eqref{eq:aniso:zoom:finite:variables}, up to the axis, and $H_n$ on the preimage of $\mathcal Q$ under~\eqref{eq:aniso:zoom:receding:variables}.

Since $\p_ZV_n$ and $\p_RW_n$ vanish at $R=0$, as $\p_z u_r$ and $\p_r u_z$ do on the axis (Section~\ref{sec:aniso:equations}), the quotients $\p_Z V_n/R$ and $\p_R W_n/R$ are bounded by the suprema of $\p_{RZ}V_n$ and $\p_{RR}W_n$ on the radial segment which joins $(0,Z,\tau)$ to $(R,Z,\tau)$; this segment again lies in the preimage of $\mathcal Q$, since $B(1)$ is convex. Thus~\eqref{eq:aniso:zoom:finite:identities} and~\eqref{eq:aniso:zoom:derivatives} give, on the preimage of $\mathcal Q$ with $R>0$,
\begin{align}
|\Omega_n|
\le \Caxi\left(\delta_n^2|\tau|^{-3/2+h}+|\tau|^{-3/2-h}\right)
\,.
\label{eq:aniso:zoom:finite:bound}
\end{align}
The bound holds at $R=0$ as well, since $\Omega_n$ is continuous across the axis. Note that the constant $\Caxi$ in~\eqref{eq:aniso:zoom:finite:bound} is the one of~\eqref{eq:aniso:zoom:derivatives}, which depends neither on $n$ nor on a time interval; this is what the decay hypothesis~\eqref{eq:aniso:ancient:decay} of Lemma~\ref{lem:aniso:comparison} requires in \emph{Step~2}.

Next, we turn to the swirl, which enters~\eqref{eq:aniso:zoom:finite:equation} through the source $F_n=S_n^2/R^2$. Since $RS_n=\delta_n\,ru_\theta$, the circulation bound~\eqref{eq:aniso:zoom:circulation} gives
\begin{subequations}
\label{eq:aniso:zoom:finite:source}
\begin{align}
|RS_n|\le C_\Gamma\delta_n
\,,
\label{eq:aniso:zoom:finite:source:a}
\end{align}
wherever the point~\eqref{eq:aniso:zoom:finite:variables} lies in $B(R_*)\times(t_*,0)$. Since $S_n$ vanishes at $R=0$, as $u_\theta$ does, $|S_n|/R$ is bounded by the supremum of $|\p_RS_n|$ along the radial segment to the axis, and in turn~\eqref{eq:aniso:zoom:derivatives} bounds $|\p_RS_n|$ by $\Caxi|\tau|^{-1-h}\le\Caxi$ for $\tau\le-1$. We thus arrive at
\begin{align}
|S_n/R|\le C
\,,
\label{eq:aniso:zoom:finite:source:b}
\end{align}
\end{subequations}
at every point of the preimage of $\mathcal Q$ with $\tau\le-1$.
Since the passage to the limit in \emph{Step~2} is carried out in the sense of distributions, what is needed of the swirl source $\p_ZF_n$ in~\eqref{eq:aniso:zoom:finite:equation} is that $F_n\to0$ in $L^1_{\rm loc}$ on $\{\tau\le-1\}$, up to the axis. Away from the axis we write $F_n=(RS_n)^2/R^4$, and~\eqref{eq:aniso:zoom:finite:source:a} gives $F_n\le C_\Gamma^2\delta_n^2\eps^{-4}\to0$ as $n\to\infty$, uniformly on compact subsets of $\{R\ge\eps\}$, for every $\eps>0$. Near the axis we write $F_n=(S_n/R)^2$, and~\eqref{eq:aniso:zoom:finite:source:b} gives $F_n\le C^2$ for $\tau\le-1$, up to $R=0$. Since $\{R<\eps\}$ has small measure within any compact set, and $\eps$ is arbitrary, the two bounds together give this convergence, and the swirl source $\p_ZF_n$ disappears in the distributional limit.

\begin{figure}[!htb]
\centering
\begin{tikzpicture}[scale=1.0]
\begin{scope}
\draw[blue!55!black,thick] (0,0) circle (2.15);
\node[font=\scriptsize,blue!55!black,anchor=south west,inner sep=1pt] at (1.44,1.46) {$B(\rho)$};
\draw[black,dashed,thick] (0,-2.55) -- (0,2.55);
\node[font=\scriptsize,anchor=south,inner sep=2pt] at (0,2.57) {$r=0$};
\fill[red!12] (-0.30,-1.42) rectangle (0.30,1.42);
\draw[red!70!black,thick] (-0.30,-1.42) rectangle (0.30,1.42);
\filldraw (0.20,0) circle (1.5pt);
\draw[thin,gray] (0.23,0.035) -- (0.48,0.16);
\node[font=\scriptsize,anchor=west,inner sep=3pt] at (0.48,0.16) {$x_n$};
\filldraw (0,0) circle (1.5pt);
\draw[thin,gray] (0.03,-0.035) -- (0.42,-0.50);
\node[font=\scriptsize,anchor=west,inner sep=2.5pt] at (0.42,-0.50) {$(0,z_n)$};
\draw[<->,thin] (-0.30,-1.68) -- (0.30,-1.68);
\node[font=\scriptsize,anchor=north,inner sep=2pt,fill=white] at (0,-1.72) {$2\lambda_n$};
\draw[<->,thin] (-0.85,-1.42) -- (-0.85,1.42);
\node[font=\scriptsize,anchor=east,inner sep=3pt] at (-0.88,0) {$2\mu_n$};
\node[font=\footnotesize,anchor=north] at (0,-2.86) {(a) $r_n/\lambda_n$ bounded};
\end{scope}
\begin{scope}[xshift=7.4cm]
\draw[blue!55!black,thick] (0,0) circle (2.15);
\node[font=\scriptsize,blue!55!black,anchor=south west,inner sep=1pt] at (1.44,1.46) {$B(\rho)$};
\draw[black,dashed,thick] (0,-2.55) -- (0,2.55);
\node[font=\scriptsize,anchor=south,inner sep=2pt] at (0,2.57) {$r=0$};
\fill[red!12] (0.94,-1.42) rectangle (1.54,1.42);
\draw[red!70!black,thick] (0.94,-1.42) rectangle (1.54,1.42);
\filldraw (1.24,0) circle (1.5pt);
\node[font=\scriptsize,anchor=west,inner sep=3.5pt] at (1.60,0.30) {$x_n$};
\draw[thin,gray,dotted] (1.24,-1.44) -- (1.24,-2.32);
\draw[thin,gray,dotted] (0,-2.18) -- (0,-2.32);
\draw[<->,thin] (0,-2.32) -- (1.24,-2.32);
\node[font=\scriptsize,anchor=north,inner sep=2pt] at (0.62,-2.34) {$r_n$};
\node[font=\footnotesize,anchor=north] at (0,-2.86) {(b) $r_n/\lambda_n\to\infty$};
\end{scope}
\end{tikzpicture}

\vspace{0.15cm}
\begin{tikzpicture}
\fill[red!12] (0,0) rectangle (0.40,0.26);
\draw[red!70!black] (0,0) rectangle (0.40,0.26);
\node[font=\footnotesize,anchor=west] at (0.52,0.13)
  {the zoom box: radial half-width $\lambda_n=(-t_n)^{1/2}$, axial half-width $\mu_n=(-t_n)^{1/2-h}\gg\lambda_n$};
\end{tikzpicture}
\caption{\footnotesize The anisotropic zoom of \emph{Step~1}, drawn in the plane through the axis. The box is long along the axis and short across it, since $h>0$. In~(a) the box is centered at the axis point $(0,z_n)$, the selected point $x_n$ lies at bounded rescaled distance from the axis, and the limit is an equation for the potential vorticity, with the Laplacian of $\RR^4$ acting on radial functions; in~(b) the box is centered at $x_n$, as in the variables~\eqref{eq:aniso:zoom:receding:variables} of \emph{Step~3}, the axis recedes, and the limit is an equation for the azimuthal vorticity on $\RR^2$, with diffusion in the radial variable alone.}
\label{fig:aniso:zoom}
\end{figure}

\emph{Step 2: the finite-axis case.}
Suppose that $r_n/\lambda_n\le A$ for some fixed $A$, as in Figure~\ref{fig:aniso:zoom}(a). In this step we prove that $\p_RV_n$, $\p_ZW_n$, and $V_n/R$ tend to zero at $\tau=-1$, uniformly on compact sets, and we derive from this a contradiction with~\eqref{eq:aniso:zoom:selected}.

The radial operator $\p_{RR}+3R^{-1}\p_R$ in~\eqref{eq:aniso:zoom:finite:equation} is the Laplacian of $\RR^4$ acting on radial functions. As in~\cite{KNSS09}, we therefore regard $\Omega_n$ as a function of $X\in\RR^4$ with $|X|=R$, so that~\eqref{eq:aniso:zoom:finite:equation} holds across the axis. These lifted variables are used only where the axis must be crossed, in the distributional form of the equation and in the weak-$*$ limits; the compactness argument for $\Omega_n$ takes place in the $(R,Z)$ half-plane. We define the drift
\begin{align}
\mathcal B_n(X,Z,\tau)
=\left(\frac X R \ V_n(R,Z,\tau),\ W_n(R,Z,\tau)\right)
\,.
\label{eq:drift:Bn:def}
\end{align}
By the radial expansions recorded in Section~\ref{sec:aniso:equations}, $\Omega_n$, $F_n$, $V_n/R$, and $W_n$ are smooth functions of $(R^2,Z,\tau)$, so that $\mathcal B_n$ is smooth in $(X,Z)$ across $X=0$. We record the identity
\begin{align*}
\nabla_X\Bigl(V_n\frac XR\Bigr)
=\frac{V_n}{R}\Bigl(\mathrm{Id}-\frac{X\otimes X}{R^2}\Bigr)+\p_RV_n\frac{X\otimes X}{R^2}
\,,
\end{align*}
and note that its trace is equal to $\p_RV_n+3V_n/R$. Since $V_n$ vanishes at $R=0$, we have $|V_n/R|\le\sup|\p_RV_n|$; combined with~\eqref{eq:aniso:zoom:derivatives}, this estimate shows that the drifts $\mathcal B_n$ are uniformly bounded and uniformly Lipschitz in space on compact time intervals,\footnote{The bound we obtain is $|\mathcal B_n|+|\nabla\mathcal B_n|\le C|\tau|^{-1/2}$ for $\tau\le-1$.} as Lemma~\ref{lem:aniso:comparison} requires of the limiting drift. Together with the divergence identity in~\eqref{eq:aniso:zoom:finite:identities}, the trace gives
\begin{align*}
\div_{X,Z}\mathcal B_n=\frac{2}{R} V_n
\,.
\end{align*}

Since $\{X=0\}$ is a null set, and since every term of~\eqref{eq:aniso:zoom:finite:equation} which carries a derivative is smooth across it, the equation
\begin{align}
\p_\tau \Omega_n+\div_{X,Z}(\mathcal B_n\Omega_n)
=\Delta_X\Omega_n+\delta_n^2\p_{ZZ}\Omega_n
+2\frac{V_n}{R}\Omega_n+\p_ZF_n+G_n
\label{eq:aniso:zoom:lifted}
\end{align}
holds almost everywhere in the lifted variables, on the preimage of $\mathcal Q$ under~\eqref{eq:aniso:zoom:finite:variables}, hence also in the sense of distributions there. Note that the force term $G_n=G_n(|X|,Z,\tau)$ in~\eqref{eq:aniso:zoom:lifted} enters only through its supremum, and satisfies $|G_n|\le CM_f\lambda_n^5\delta_n\to0$ on that set, as recorded after~\eqref{eq:interior:force:scaled}.

Let $K\subset\{R>0\}$ be a closed rectangle in the half-plane. The second identity in~\eqref{eq:aniso:zoom:finite:identities} and the bounds in~\eqref{eq:aniso:zoom:derivatives} may be used to estimate the first spatial derivatives of $\Omega_n$ in $L^\infty(K)$, uniformly in $n$ on compact time intervals within $(-\infty,-1]$.
Using~\eqref{eq:aniso:zoom:finite:equation}, we may bound $\p_\tau\Omega_n$ in $W^{-2,p}(\mathrm{int}\,K)$, uniformly on compact time intervals, for any $p\in[1,\infty)$, since every term in~\eqref{eq:aniso:zoom:finite:equation} other than $\p_\tau\Omega_n$ is at most two derivatives of a function bounded on $K$ uniformly in $n$.\footnote{Here we use the bounds~\eqref{eq:aniso:zoom:derivatives},~\eqref{eq:aniso:zoom:finite:bound}, the first derivative bounds on $\Omega_n$ just established, the estimate~\eqref{eq:aniso:zoom:finite:source:a}, and the uniform vanishing of $G_n$.} A standard compactness argument then gives, for every $\eta>0$, a constant $C_\eta$ (which depends also on $K$ and $p$) such that
\begin{align*}
\norm{F}_{C^0(K)}\le\eta\norm{F}_{C^{0,1}(K)}+C_\eta\norm{F}_{W^{-2,p}(\mathrm{int}\,K)}
\,,
\end{align*}
for all $F \in C^{0,1}(K)$.
Applying this inequality to $\Omega_n(\tau)-\Omega_n(\tau')$, we deduce that the maps $\tau\mapsto\Omega_n(\tau)$ are equicontinuous with values in $C^0(K)$. By Arzel\`a-Ascoli and a diagonal extraction, we arrive at
\begin{align}
\Omega_n\to \Omega_\infty
\quad\mbox{locally uniformly on }\{R>0,\ \tau\le-1\}
\,.
\label{eq:aniso:zoom:finite:compactness}
\end{align}
The endpoint $\tau=-1$ is included because every estimate above holds on each closed interval $[-T,-1]$.

Next, we extend $\Omega_\infty$ by zero to the null set $\{X=0\}$, and claim that the convergence also holds across the axis, in the sense that
\begin{align}
\Omega_n\to\Omega_\infty
\quad\mbox{in } L^p_{\rm loc}\bigl(\RR^4\times\RR\times(-\infty,-1]\bigr)
\quad\mbox{for every } 1\le p<\infty
\label{eq:aniso:zoom:finite:strong}
\,.
\end{align}
Indeed, for $\tau\le-1$ the bound~\eqref{eq:aniso:zoom:finite:bound} gives $|\Omega_n|\le C$ uniformly in $n$, up to the axis, and hence $|\Omega_\infty|\le C$ as well by~\eqref{eq:aniso:zoom:finite:compactness}, so that the tube $\{|X|<\eps\}$, whose measure is of order $\eps^4$ (at fixed $Z$ and $\tau$), contributes little to the $L^p$ norm of $\Omega_n-\Omega_\infty$ on any compact set, uniformly in $n$; off the tube, the convergence~\eqref{eq:aniso:zoom:finite:compactness} is uniform on compact sets. We used the same reasoning in \emph{Step~1}, for the convergence $F_n\to0$ in $L^1_{\rm loc}$.

We next pass to the limit in~\eqref{eq:aniso:zoom:lifted}, and show that $\Omega_\infty$ solves~\eqref{eq:aniso:zoom:finite:limit} below in the weak sense.
For this purpose, the bounds~\eqref{eq:aniso:zoom:derivatives} allow us to extract weak-$*$ limits $V$ and $W$ of $V_n$ and $W_n$, on compact subsets of the lifted space-time domain. Since $X/R$ is a bounded function independent of $n$, the drifts $\mathcal B_n$ defined in~\eqref{eq:drift:Bn:def} converge weakly-$*$ to $\mathcal B=(VX/R,W)$. These drifts are uniformly Lipschitz in space on compact time intervals, so that we simultaneously extract the weak-$*$ limit of the gradients $\nabla\mathcal B_n$, which is the distributional gradient of $\mathcal B$. The constants in~\eqref{eq:aniso:zoom:derivatives} do not depend on the compact spatial set, so that $|\mathcal B|+|\nabla\mathcal B|\le C_J$ almost everywhere on $\RR^4\times\RR\times J$, for every compact $J\subset(-\infty,-1)$. By Fubini, $\mathcal B(\cdot,\tau)\in W^{1,\infty}(\RR^4\times\RR)$ for almost every $\tau$, with the same bound, so that~\eqref{eq:aniso:comparison:drift} holds with $\Lambda_J\le C_J$. Moreover, since $|X|^{-1}$ is locally integrable in four dimensions, $\varphi/R$ is integrable for every $\varphi\in C^\infty_c(\RR^4\times\RR\times(-\infty,-1))$, so that the weak-$*$ convergence of $V_n$ gives $V_n/R\to V/R$ in distributions. Passing to the limit in $\div_{X,Z}\mathcal B_n=2V_n/R$ yields $\div_{X,Z}\mathcal B=2V/R$.
The weak-$*$ convergence of $\mathcal B_n$ against the strong convergence~\eqref{eq:aniso:zoom:finite:strong} of $\Omega_n$, together with the uniform bound $|V_n/R|\le\sup|\p_RV_n|$, allows us to pass to the limit in the products in~\eqref{eq:aniso:zoom:lifted}, while $\delta_n\to0$, $G_n\to0$ uniformly, and $F_n\to0$ in $L^1_{\rm loc}$ remove the remaining terms. In view of $\div_{X,Z}\mathcal B=2V/R$, the limit is the equation of Lemma~\ref{lem:aniso:comparison}, in the divergence form of its footnote, namely
\begin{align}
\label{eq:aniso:zoom:finite:limit}
\p_\tau \Omega_\infty+\mathcal B\cdot\nabla \Omega_\infty=\Delta_X\Omega_\infty
\quad\mbox{on }\RR^4\times\RR\times(-\infty,-1)
\,.
\end{align}

Since $\delta_n^2|\tau|^{-3/2+h}\le\delta_n^2$ for $\tau\le-1$, the bound~\eqref{eq:aniso:zoom:finite:bound} implies
\begin{align*}
|\Omega_\infty(X,Z,\tau)|\le \Caxi|\tau|^{-3/2-h}
\,.
\end{align*}
Thus $\Omega_\infty$ is bounded on $\RR^4\times\RR\times J$ for every compact $J\subset(-\infty,-1)$, solves~\eqref{eq:aniso:zoom:finite:limit} in distributions with a drift satisfying~\eqref{eq:aniso:comparison:drift}, and obeys~\eqref{eq:aniso:ancient:decay} with $\kappa=3/2+h$. The last assertion of Lemma~\ref{lem:aniso:comparison}, with $(m,d)=(5,4)$, gives
\[
\Omega_\infty=0 \,.
\]
By~\eqref{eq:aniso:zoom:finite:compactness}, $\Omega_\infty$ is continuous on $\{R>0,\ \tau\le-1\}$, so that it vanishes at every point of this set, the endpoint $\tau=-1$ being included.

At $\tau=-1$, the bounds~\eqref{eq:aniso:zoom:derivatives} make $V_n(\cdot,-1)$ and $W_n(\cdot,-1)$ bounded in $C^2$ on every compact subset of the closed half-plane, with a constant which does not depend on the compact set. Passing to a further subsequence, we obtain $V_n(\cdot,-1)\to V^0$ and $W_n(\cdot,-1)\to W^0$ in $C^1$ on compact subsets, with $V^0$ and $W^0$ bounded on the entire half-plane. Passing to the limit in the first identity in~\eqref{eq:aniso:zoom:finite:identities}, we obtain
\begin{align*}
\p_R V^0+\frac{V^0}{R}+\p_Z W^0=0
\,,
\end{align*}
on $\{R>0\}$. Since $\Omega_n(\cdot,-1)\to0$ off the axis, by~\eqref{eq:aniso:zoom:finite:compactness} and $\Omega_\infty=0$, and since $\delta_n\to0$ while $\p_ZV_n(\cdot,-1)$ is bounded, the second identity in~\eqref{eq:aniso:zoom:finite:identities} gives
\begin{align*}
\p_RW^0=0
\,,
\end{align*}
on $\{R>0\}$.
In contrast to~\cite{CIVEulerLength}, the vorticity identity yields only $\p_RW^0=0$, not a curl free limit; instead, the boundedness of $V^0$ replaces Liouville's theorem.
Since $W^0$ does not depend on $R$, we may rewrite the divergence identity for $(V^0,W^0)$ as $\p_R(RV^0)=-R\,\p_ZW^0(Z)$. Integrating this from the axis, where $RV^0$ vanishes because $V^0$ is bounded, we obtain
\begin{align*}
V^0(R,Z)=-\frac R2\p_ZW^0(Z)
\,.
\end{align*}
The boundedness of $V^0$ for all $R\ge0$ forces $\p_ZW^0=0$, and then also $V^0=0$, on $\{R>0\}$, hence on the closed half-plane by continuity. Since the convergence to $V^0$ and $W^0$ holds in $C^1$ on compact subsets of the closed half-plane, we conclude that, uniformly on every such compact subset,
\begin{align}
\p_RV_n(\cdot,-1)\to 0
\,,\qquad
\p_ZW_n(\cdot,-1)\to 0
\,,\qquad
V_n(\cdot,-1)/R\to 0
\,.
\label{eq:aniso:zoom:finite:endpoint}
\end{align}
The convergences hold up to the axis; for the last one this is because $V_n(R,Z,-1)/R=\int_0^1\p_RV_n(\sigma R,Z,-1)\,d\sigma$ for $R>0$, with the right side continuous at $R=0$.

With~\eqref{eq:aniso:zoom:finite:endpoint} in hand, we return to the points $x_n$ selected in~\eqref{eq:aniso:zoom:selected}. In the variables~\eqref{eq:aniso:zoom:finite:variables}, the point $(x_n,t_n)$ has coordinates $(R,Z,\tau)=(r_n/\lambda_n,0,-1)$, and $(r_n/\lambda_n,0)$ lies in the compact set $[0,A]\times\{0\}$ for every $n$. Since $\p_R=\lambda_n\p_r$, $\p_Z=\mu_n\p_z$, $R=r/\lambda_n$, and $\lambda_n^2=\lambda_n\delta_n\mu_n=(-t_n)$, the definitions~\eqref{eq:aniso:zoom:fields} and the chain rule show that the left side of~\eqref{eq:aniso:zoom:selected} equals
\begin{align*}
(-t_n)\Bigl(|\p_ru_r|+\Bigl|\frac{u_r}{r}\Bigr|+|\p_zu_z|\Bigr)(x_n,t_n)
=\Bigl(|\p_RV_n|+\Bigl|\frac{V_n}{R}\Bigr|+|\p_ZW_n|\Bigr)\Bigl(\frac{r_n}{\lambda_n},0,-1\Bigr)
\,,
\end{align*}
with both quotients extended continuously to the axis. Since the convergence~\eqref{eq:aniso:zoom:finite:endpoint} is uniform on $[0,A]\times\{0\}$, the right side of the last display tends to zero along the subsequence, while its left side is bounded from below by $c_0>0$. The contradiction ansatz~\eqref{eq:aniso:zoom:selected} thus fails in the finite-axis case.

\emph{Step 3: the receding-axis case.}
We now turn to the other alternative of \emph{Step~1}: suppose that $A_n=r_n/\lambda_n\to\infty$, as in Figure~\ref{fig:aniso:zoom}(b). In this step we prove that $\p_RV_n$ and $\p_ZW_n$, now taken in the recentered variables~\eqref{eq:aniso:zoom:receding:variables}, tend to zero at the selected point, and that the remaining term $(-t_n)u_r/r$ is small there by~\eqref{eq:aniso:bounds} alone; we derive from this a contradiction with~\eqref{eq:aniso:zoom:selected}.

The argument follows \emph{Step~2}, with the azimuthal vorticity in place of the potential vorticity and no axis to cross. We now zoom in about the selected point $x_n$ itself, in the variables $(R,Z,\tau)$ defined by
\begin{align}
r=r_n+\lambda_nR=\lambda_n(A_n+R)
\,,\qquad z=z_n+\mu_nZ
\,,\qquad t=(-t_n)\tau
\,.
\label{eq:aniso:zoom:receding:variables}
\end{align}
We define $V_n,W_n,S_n$ as in~\eqref{eq:aniso:zoom:fields}, with the same prefactors and with the point in~\eqref{eq:aniso:zoom:receding:variables} as argument. The rescaled azimuthal vorticity is
\begin{align}
\Theta_n(R,Z,\tau) =\lambda_n^2\delta_n \, \omega_\theta(r,z,t)
=\delta_n^2 \p_Z V_n(R,Z,\tau) -\p_R W_n(R,Z,\tau)
\,,
\label{eq:aniso:zoom:receding:vorticity}
\end{align}
where $(r,z,t)$ is the point~\eqref{eq:aniso:zoom:receding:variables}.
Since $\lambda_n,\mu_n\to0$ by~\eqref{eq:aniso:zoom:lengths}, $A_n\to\infty$ by assumption, and $|x_n|<\rho<R_*$, the preimage of $\mathcal Q$ under~\eqref{eq:aniso:zoom:receding:variables}, whose axis sits at $R=-A_n$, contains any given compact subset of $\RR^2\times(-\infty,-1]$ for all $n$ large depending on that subset; so does the preimage of $B(R_*)\times(t_*,0)$, on which~\eqref{eq:aniso:zoom:circulation} holds. The bounds~\eqref{eq:aniso:zoom:derivatives} are unchanged, since $\p_R=\lambda_n\p_r$ and $\p_Z=\mu_n\p_z$ as in~\eqref{eq:aniso:zoom:finite:variables}, and since~\eqref{eq:aniso:bounds} holds at every point of $\mathcal Q$. The divergence equation~\eqref{eq:aniso:scalar:nse:div} and the azimuthal vorticity equation~\eqref{eq:aniso:theta} become
\begin{subequations}
\label{eq:aniso:zoom:receding}
\begin{align}
\p_RV_n+\frac{V_n}{A_n+R}+\p_ZW_n&=0
\,,
\label{eq:aniso:zoom:receding:div}
\end{align}
and
\begin{align}
&\p_\tau\Theta_n+V_n\p_R\Theta_n+W_n\p_Z\Theta_n
-\frac{V_n}{A_n+R}\Theta_n
\notag\\
&\quad =\left(\p_{RR}+\frac1{A_n+R}\p_R
+\delta_n^2\p_{ZZ}-\frac1{(A_n+R)^2}\right)\Theta_n
+\frac1{A_n+R}\p_Z(S_n^2)+H_n
\,,
\label{eq:aniso:zoom:receding:equation}
\end{align}
\end{subequations}
at every point of the preimage of $\mathcal Q$ under~\eqref{eq:aniso:zoom:receding:variables} with $A_n+R>0$, and with $H_n$ given by~\eqref{eq:interior:force:scaled:b}.
By~\eqref{eq:aniso:zoom:receding:vorticity} and~\eqref{eq:aniso:zoom:derivatives}, on the preimage of $\mathcal Q$ under~\eqref{eq:aniso:zoom:receding:variables},
\begin{align}
|\Theta_n|
\le \Caxi\left(\delta_n^2|\tau|^{-1+h}+|\tau|^{-1-h}\right)
\,.
\label{eq:aniso:zoom:receding:bound}
\end{align}
This plays the role of~\eqref{eq:aniso:zoom:finite:bound} in \emph{Step~2}.

Since $(A_n+R)S_n=\delta_n\,ru_\theta$ (see~\eqref{eq:aniso:zoom:fields:c} and~\eqref{eq:aniso:zoom:receding:variables}), the circulation bound~\eqref{eq:aniso:zoom:circulation} gives
\[
|S_n|\le C_\Gamma \frac{\delta_n}{A_n+R}
\]
wherever the point defined in~\eqref{eq:aniso:zoom:receding:variables} lies in $B(R_*)\times(t_*,0)$, and $A_n+R>0$.\footnote{We appeal here to the circulation bound rather than to the $S_n$ estimate in~\eqref{eq:aniso:zoom:derivatives}, so that the bounds on $u_\theta$ enter this section in one place only; see Remark~\ref{rem:interior:meridional}.}

Let $K\subset\RR^2$ be a closed rectangle. The identity~\eqref{eq:aniso:zoom:receding:vorticity} and the bounds in~\eqref{eq:aniso:zoom:derivatives} may be used to estimate the first spatial derivatives of $\Theta_n$ in $L^\infty(K)$, uniformly in $n$ on compact time intervals within $(-\infty,-1]$. Using~\eqref{eq:aniso:zoom:receding:equation}, we may bound $\p_\tau\Theta_n$ in $W^{-2,p}(\mathrm{int}\,K)$, uniformly on compact time intervals, for any $p\in[1,\infty)$, since every term in~\eqref{eq:aniso:zoom:receding:equation} other than $\p_\tau\Theta_n$ is at most two derivatives of a function bounded on $K$ uniformly in $n$.\footnote{Here we use the bounds~\eqref{eq:aniso:zoom:derivatives} and~\eqref{eq:aniso:zoom:receding:bound}, the first derivative bounds on $\Theta_n$ just established, the uniform vanishing of $H_n$, and the bound on $S_n$ just obtained, applied to the swirl source written as $\p_Z(S_n^2/(A_n+R))$, which is possible since $A_n+R$ does not depend on $Z$.} The compactness argument of \emph{Step~2}, with no axis to cross, then gives
\[
\Theta_n\to\Theta \qquad \mbox{locally uniformly on } \RR^2\times(-\infty,-1] \,.
\]

Our next goal is to pass to the limit $n\to\infty$ in~\eqref{eq:aniso:zoom:receding} and~\eqref{eq:aniso:zoom:receding:bound}, and to show that $\Theta$ solves the limiting equation~\eqref{eq:aniso:zoom:receding:limit:a} below. For this purpose, since by~\eqref{eq:aniso:zoom:derivatives} the fields $V_n$ and $W_n$ are uniformly bounded and uniformly Lipschitz in space on compact time intervals, we extract weak-$*$ limits $V,W$ of $V_n,W_n$ on compact subsets of $\RR^2\times(-\infty,-1]$, together with the weak-$*$ limits of their spatial gradients, which are the distributional gradients of $V$ and $W$.
Three terms of~\eqref{eq:aniso:zoom:receding:equation} disappear in the limit because the axis recedes. On a compact spatial set the factors $(A_n+R)^{-1}$ and $(A_n+R)^{-2}$ tend uniformly to zero as $A_n\to\infty$, while $\Theta_n$ and $\p_R\Theta_n$ are bounded there, uniformly in $n$ on compact time intervals; hence the viscous terms $(A_n+R)^{-1}\p_R\Theta_n$ and $(A_n+R)^{-2}\Theta_n$ tend uniformly to zero. The swirl source, written as $\p_Z(S_n^2/(A_n+R))$, tends to zero in distributions, since $S_n^2/(A_n+R)\le C_\Gamma^2\delta_n^2/(A_n+R)^3$ by the bound on $S_n$ above, and the right side vanishes uniformly as $n\to\infty$.
The term $\delta_n^2\p_{ZZ}\Theta_n$ likewise tends to zero in distributions because $\delta_n\to0$. Recall from the discussion after~\eqref{eq:interior:force:scaled} that the force term $H_n\to0$ uniformly. The stretching term $-(A_n+R)^{-1}V_n\Theta_n$ on the left side of~\eqref{eq:aniso:zoom:receding:equation} is not discarded in the limit. By~\eqref{eq:aniso:zoom:receding:div} it equals $\Theta_n\div_{R,Z}(V_n,W_n)$, so that on the left side of~\eqref{eq:aniso:zoom:receding:equation} we have a perfect divergence
\begin{align*}
V_n\p_R\Theta_n+W_n\p_Z\Theta_n-\frac{V_n}{A_n+R}\Theta_n
=\div_{R,Z}\bigl((V_n,W_n)\Theta_n\bigr)
\,.
\end{align*}

Passing to the limit, using the weak-$*$ convergence of $(V_n,W_n)$ together with the strong convergence of $\Theta_n$, and with $V_n$ bounded and $A_n\to\infty$ in~\eqref{eq:aniso:zoom:receding:div}, we obtain
\begin{subequations}
\label{eq:aniso:zoom:receding:limit}
\begin{align}
\p_\tau\Theta+V\p_R\Theta+W\p_Z\Theta&=\p_{RR}\Theta
\label{eq:aniso:zoom:receding:limit:a}
\,,\\
\p_RV+\p_ZW&=0 \,,
\end{align}
on $\RR^2\times(-\infty,-1)$. For $\tau\le-1$, the bound
\begin{align}
|\Theta(R,Z,\tau)|\le \Caxi|\tau|^{-1-h}
\end{align}
\end{subequations}
is inherited from~\eqref{eq:aniso:zoom:receding:bound}, since $\delta_n^2|\tau|^{-1+h}\le\delta_n^2$ in that range. The terms other than $\p_\tau\Theta_n$ on the left side of~\eqref{eq:aniso:zoom:receding:equation} converge in distributions to $\div_{R,Z}\bigl((V,W)\Theta\bigr)$, and the limiting drift is divergence free, so that the transport term in~\eqref{eq:aniso:zoom:receding:limit:a} is the one of the footnote of Lemma~\ref{lem:aniso:comparison}.

We are now in a position to apply Lemma~\ref{lem:aniso:comparison}.
The limiting field $\Theta$ is continuous on $\RR^2\times(-\infty,-1]$, being a locally uniform limit, and is bounded on $\RR^2\times J$ for every compact $J\subset(-\infty,-1)$. Since the bounds in~\eqref{eq:aniso:zoom:derivatives} do not depend on the compact spatial set, the limiting drift $(V,W)$ lies, as in \emph{Step~2}, in $W^{1,\infty}(\RR^2)$ for almost every $\tau$, and satisfies~\eqref{eq:aniso:comparison:drift} on every compact time interval. The last assertion of Lemma~\ref{lem:aniso:comparison}, applied to~\eqref{eq:aniso:zoom:receding:limit} with $(m,d)=(2,1)$ and $\kappa=1+h$, therefore gives
\[
\Theta=0
\]
on $\RR^2\times(-\infty,-1]$, by continuity. In particular, $\Theta(\cdot,-1)=0$.

At $\tau=-1$, the bounds~\eqref{eq:aniso:zoom:derivatives} make $V_n(\cdot,-1)$ and $W_n(\cdot,-1)$ bounded in $C^2$ on every compact subset of $\RR^2$. Passing to a further subsequence, we obtain $V_n(\cdot,-1)\to V^0$ and $W_n(\cdot,-1)\to W^0$ in $C^1$ on compact subsets. The limits $V^0$ and $W^0$ are bounded on $\RR^2$ by the constant $\Caxi$ of~\eqref{eq:aniso:zoom:derivatives}, which does not depend on the compact set.
Passing to the limit in the divergence identity~\eqref{eq:aniso:zoom:receding:div}, in which $V_n/(A_n+R)\to0$ uniformly on compact sets since $V_n$ is bounded and $A_n\to\infty$, we obtain
\begin{align*}
\p_RV^0+\p_ZW^0=0
\,.
\end{align*}
Since $\Theta_n(\cdot,-1)\to\Theta(\cdot,-1)=0$ locally uniformly, and since $\delta_n\to0$ while $\p_ZV_n$ is bounded, the identity~\eqref{eq:aniso:zoom:receding:vorticity} gives
\begin{align*}
\p_RW^0=0
\,,
\end{align*}
so that $W^0=W^0(Z)$. Integrating the divergence identity $\p_RV^0+\p_ZW^0=0$ in $R$, we obtain
\begin{align*}
V^0(R,Z)=-R\,\p_ZW^0(Z)+V^0(0,Z) \,.
\end{align*}
As in \emph{Step~2}, the boundedness of $V^0$ for all $R\in\RR$ forces $\p_ZW^0=\p_RV^0=0$. Since the convergence at $\tau=-1$ holds in $C^1$ on compact subsets, we obtain
\begin{align}
\p_RV_n(0,0,-1)\to 0
\,,\qquad
\p_ZW_n(0,0,-1)\to 0
\,.
\label{eq:aniso:zoom:receding:endpoint}
\end{align}

To conclude, we compare~\eqref{eq:aniso:zoom:receding:endpoint} with the selection~\eqref{eq:aniso:zoom:selected}. In the variables of~\eqref{eq:aniso:zoom:receding:variables} the point $(x_n,t_n)$ is $(0,0,-1)$, and the chain rule gives
\begin{align*}
(-t_n)\,\p_ru_r(x_n,t_n)=\p_RV_n(0,0,-1)
\,,\qquad
(-t_n)\,\p_zu_z(x_n,t_n)=\p_ZW_n(0,0,-1)
\,.
\end{align*}
Thus, the first and the third term on the left side of~\eqref{eq:aniso:zoom:selected} tend to zero by~\eqref{eq:aniso:zoom:receding:endpoint}. For the second term no limiting argument is needed: since $A_n\to\infty$, we have $r_n>0$ for all $n$ large, and~\eqref{eq:aniso:bounds} gives
\begin{align}
(-t_n)\left|\frac{u_r(x_n,t_n)}{r_n}\right|
\le C\frac{\lambda_n}{r_n}=\frac{C}{A_n}\to0
\,.
\label{eq:aniso:zoom:receding:quotient}
\end{align}
The left side of~\eqref{eq:aniso:zoom:selected} thus tends to zero along the subsequence, which contradicts the lower bound $c_0>0$ in the receding-axis case as well, thereby concluding the proof of~\eqref{eq:aniso:small}.
\end{proof}

\section{The local regularity criterion}
\label{sec:aniso:closure}

In this section we show that the relative smallness provided by Proposition~\ref{prop:aniso:small}, together with the regular annulus of Lemma~\ref{lem:aniso:annulus}, makes $(0,0)$ a regular point, which proves Theorem~\ref{thm:aniso:main}. 

\begin{lemma}[{\bf Smallness of the meridional quantities implies regularity}]
\label{lem:aniso:closure}
Let $(u,\pi)$ be an axisymmetric suitable weak solution of~\eqref{eq:nse:forced} on $\mathcal Q$, smooth on compact subsets of $\mathcal Q$, with an axisymmetric force satisfying~\eqref{eq:interior:force:c-two}; the bounds~\eqref{eq:aniso:bounds} are not assumed. Suppose that
\begin{align}
\lim_{t\uparrow0}(-t)G_{\rho_0}(t)=0
\label{eq:aniso:closure:small}
\end{align}
holds for some $\rho_0\in(0,1)$, with $G_{\rho_0}$ defined as in~\eqref{eq:aniso:G}. Then $(0,0)$ is a regular point.
\end{lemma}

\begin{proof}[Proof of Lemma~\ref{lem:aniso:closure}]
We fix a regular annulus for $u$ inside $B(\rho_0)$, as in Lemma~\ref{lem:aniso:annulus} applied with $R_1=0$ and $R_0=\rho_0$: on the cylindrical annulus $\{R_-<|x|<R_+\}\times(t_0,0)$ the velocity $u$, together with $\nabla u$ and $\nabla^2u$, is bounded. We then fix radii $R_-<\rho<R_*<R_+$ and a smooth radial cutoff $\chi$ supported in $B(R_*)$, equal to one on $B(\rho)$, whose derivatives are supported in the cylindrical annulus. Throughout the proof, $C$ denotes a constant which depends only on $\chi$, on $M_f$, on the $u$ bounds on the cylindrical annulus, but never on the parameter $\eta$ chosen at the end of the proof. In particular, every term in which a derivative falls on $\chi$ is bounded by such a constant, since it is supported in the cylindrical annulus. Lastly, norms without a specified domain are taken on $\RR^3$, the localized fields being extended by zero.

We write $G(t)=G_{\rho_0}(t)$, and for each $\eta>0$ we fix $t_\eta\in(t_0,0)$ so close to zero that
\begin{align}
G(t)\le\frac{\eta}{(-t)}
\qquad\mbox{on }(t_\eta,0)
\,.
\label{eq:aniso:G:bound}
\end{align}
Since $B(R_*)\subset B(\rho_0)$, the four quantities in~\eqref{eq:aniso:G} are bounded by $G$ on $B(R_*)$; in particular $|u_r/r|\le G$ there. The azimuthal velocity $u_\theta$, moreover, vanishes on the axis, and is bounded on $\partial B(R_*)\times(t_0,0)$ and also on the initial time slice $B(R_*)\times\{t_\eta\}$. By the swirl equation~\eqref{eq:aniso:scalar:nse:swirl}, the function $\psi=(-t)^\eta u_\theta$ satisfies
\[
\p_t\psi+b\cdot\nabla\psi-\Delta\psi
+\Bigl(r^{-2}+\frac{u_r}{r}+\frac{\eta}{(-t)}\Bigr)\psi
=(-t)^\eta f_\theta
\,,
\]
where $b=u_re_r+u_ze_z$ is the meridional velocity. By~\eqref{eq:aniso:G:bound}, the zeroth order coefficient is therefore at least $r^{-2}\ge0$ on $B(R_*)\times(t_\eta,0)$, and the source is bounded by $M_f$. Since $\Delta=\p_r^2+r^{-1}\p_r+\p_z^2$ on axisymmetric scalars, this equation is of the form~\eqref{eq:aniso:axis:pde} with $k=1$. The coefficient $r^{-2}$ is singular on the axis, but $\psi$ vanishes there, and for every $s\in(t_\eta,0)$ it is continuous on $\overline{B(R_*)}\times[t_\eta,s]$ and smooth off the axis. As in the proof of Lemma~\ref{lem:aniso:annulus}, appealing to Lemma~\ref{lem:aniso:axis} we obtain a bound for $|\psi|$, uniformly in $s$, by its values at time $t_\eta$ and on $\partial B(R_*)\times[t_\eta,0)$, plus $M_f|t_\eta|$. That is, there is a constant $C_\eta>0$ such that $|\psi|\le C_\eta$ on $B(R_*)\times(t_\eta,0)$, and 
\begin{align}
W(t):=\norm{u_\theta(\cdot,t)}_{L^\infty(B(R_*))}
\le C_\eta(-t)^{-\eta}
\,,\qquad t_\eta<t<0
\,.
\label{eq:aniso:closure:w}
\end{align}

We let
\[
Y(t)=\norm{\chi\,\omega(\cdot,t)}_{L^2}^2
\,,\qquad
M(t)=\norm{\chi\,\nabla\omega(\cdot,t)}_{L^2}^2
\,,
\]
where $\omega=\curl u$. Taking the curl of~\eqref{eq:nse:forced}, we obtain
\[
\p_t\omega+u\cdot\nabla\omega-\Delta\omega=\omega\cdot\nabla u+\curl f
\,,
\]
and we test this equation against $\chi^2\omega$. The transport and diffusion terms then produce $\frac12 \frac{d}{dt} Y$ and $M$, up to cutoff errors which are bounded by $C$. The force term, in turn, is bounded by $C\norm{\curl f}_{L^\infty}Y^{1/2}\le C(Y+1)$, and we arrive at
\begin{align}
\frac12 \frac{d}{dt} Y +M\le\int\chi^2(\omega\cdot\nabla u)\cdot\omega\,dx+C(Y+1)
\,.
\label{eq:aniso:closure:tested}
\end{align}
Bounding the stretching term requires two elliptic bounds, which we record next. Since $\div b=0$ and $\curl b=\omega_\theta e_\theta$, the identities $\|\nabla v\|_{L^2}^2=\|\curl v\|_{L^2}^2+\|\div v\|_{L^2}^2$ and $\|\nabla^2v\|_{L^2}^2=\|\nabla\curl v\|_{L^2}^2+\|\nabla\div v\|_{L^2}^2$, valid for smooth compactly supported $v$, applied to $\chi u$ and to $\chi b$, give
\begin{subequations}
\label{eq:aniso:closure:elliptic}
\begin{align}
\norm{\nabla(\chi u)}_{L^2}+\norm{\chi\nabla b}_{L^2}
&\le C(Y+1)^{1/2}
\,,\\
\norm{\chi\nabla^2b}_{L^2}
&\le C(M+1)^{1/2}
\,.
\end{align}
\end{subequations}
The second inequality also uses $|\nabla(\omega_\theta e_\theta)|\le|\nabla\omega|$, which follows from the Cartesian gradient identity for an axisymmetric vector field,
\begin{align}
|\nabla\omega|^2
=\sum_{\circ\in\{r,\theta,z\}}\bigl(|\p_r\omega_\circ|^2+|\p_z\omega_\circ|^2\bigr)
+\frac{\omega_r^2+\omega_\theta^2}{r^2}
\,.
\label{eq:aniso:closure:gradient}
\end{align}
In particular, we note that the axis terms $\omega_r^2/r^2$ and $\omega_\theta^2/r^2$ are part of the dissipation.

In the cylindrical frame, with $\omega_r=-\p_zu_\theta$ and $\omega_z=\p_ru_\theta+u_\theta/r$, the stretching term reads
\begin{align}
(\omega\cdot\nabla u)\cdot\omega
={}&(\p_r u_r)\omega_r^2
+\frac{u_r}{r}\omega_\theta^2
+(\p_z u_z)\omega_z^2
+(\p_z u_r+\p_r u_z)\omega_r\omega_z
-2\frac{u_\theta}{r}\omega_r\omega_\theta
\,.
\label{eq:aniso:closure:stretch}
\end{align}
After multiplication by $\chi^2$ and integration, the first three terms on the right side of~\eqref{eq:aniso:closure:stretch}, and the term containing $\p_zu_r$, are all bounded by $C\, G\, Y$, their coefficients being the four quantities in the definition in~\eqref{eq:aniso:G}. For the term containing $\p_ru_z$ we integrate by parts in $z$ using that $\omega_r=-\p_zu_\theta$, and arrive at 
\begin{align}
I:=\int\chi^2(\p_r u_z)\omega_r\omega_z\,dx
=-\int\chi^2(\p_r u_z)(\p_zu_\theta)\omega_z\,dx
=\int u_\theta\,\p_z\bigl(\chi^2(\p_r u_z)\omega_z\bigr)\,dx
\,.
\label{eq:aniso:closure:mixed}
\end{align}
The bound~\eqref{eq:aniso:closure:w} allows us to estimate the swirl in $L^\infty$. The term containing $\p_z\chi$ is bounded by $C$, as noted at the start of the proof, while for the other two terms, \eqref{eq:aniso:closure:elliptic} together with $|\p_z\p_r u_z|\le|\nabla^2b|$ and $|\p_r u_z|\le|\nabla b|$ gives
\begin{align}
|I|
&\le C+W\norm{\chi\p_z\p_r u_z}_{L^2}\norm{\chi\omega_z}_{L^2}
+W\norm{\chi\p_r u_z}_{L^2}\norm{\chi\p_z\omega_z}_{L^2}
\notag\\
&\le C+CW(M+1)^{1/2}Y^{1/2}
+CW(Y+1)^{1/2}M^{1/2}
\notag\\
&\le\tfrac14M+C(W^2+1)(Y+1)
\,.
\label{eq:aniso:closure:mixed:bound}
\end{align}
Returning to~\eqref{eq:aniso:closure:stretch}, it thus remains to bound the last term, which contains $\omega_r/r$. Using~\eqref{eq:aniso:closure:gradient} we get
\begin{align}
2\left|\int\chi^2\frac{u_\theta}{r}\omega_r\omega_\theta\,dx\right|
\le2W\norm{\chi\omega_r/r}_{L^2}\norm{\chi\omega_\theta}_{L^2}
\le2W M^{1/2}Y^{1/2}
\le\tfrac14M+4W^2Y
\,.
\label{eq:aniso:closure:angular:bound}
\end{align}
Inserting~\eqref{eq:aniso:closure:mixed:bound} and~\eqref{eq:aniso:closure:angular:bound} in~\eqref{eq:aniso:closure:tested}, absorbing $\frac12M$ into the left side, we have thus shown that there exists $C_*\ge1$, which depends only on $\chi$, on $M_f$, and on the bounds for $u$ on the cylindrical annulus, such that
\begin{align}
\frac{d}{dt} Y+M\le C_*\bigl(G+W^2+1\bigr)(Y+1)
\,.
\label{eq:aniso:closure:energy}
\end{align}

At this point we fix $\eta=1/(16C_*)\le1/16$; this choice is permissible because $C_*$ was fixed before $\eta$ was chosen. Combining~\eqref{eq:aniso:closure:energy} with the bounds~\eqref{eq:aniso:G:bound} and~\eqref{eq:aniso:closure:w}, Gr\"onwall's inequality gives
\[
Y(t)+1\le\bigl(Y(t_\eta)+1\bigr)
\exp\left(\int_{t_\eta}^t\Bigl(\frac{C_*\eta}{(-s)}+\frac{C_*C_\eta^2}{(-s)^{2\eta}}+C_*\Bigr)ds\right)
\le C_\eta'(-t)^{-C_*\eta}
\]
for all $t\in(t_\eta,0)$, for a suitable constant $C_\eta'>0$. Note that $Y(t_\eta)$ is finite since $u$ is smooth at time $t_\eta$. The first inequality in~\eqref{eq:aniso:closure:elliptic} and the Sobolev inequality then give
\[
\norm{u(\cdot,t)}_{L^6(B(\rho))}
\le\norm{\chi u(\cdot,t)}_{L^6(\RR^3)}
\le C\norm{\nabla(\chi u)(\cdot,t)}_{L^2}
\le C_\eta''(-t)^{-C_*\eta/2}
\]
for all $t\in(t_\eta,0)$. Since $2C_*\eta=1/8<1$, this bound implies the local Ladyzhenskaya-Prodi-Serrin condition
\[
u\in L^4(t_\eta,0;L^6(B(\rho)))
\,,
\]
at the borderline exponent pair $3/6+2/4=1$. Consequently $\norm{u}_{L^4(-r^2,0;L^6(B(r)))}$ tends to zero as $r\downarrow0$, for $r\le\min\{\rho,(-t_\eta)^{1/2}\}$. The interior criterion of Gustafson, Kang, and Tsai~\cite[Theorem~1.1(i), with $p_*=6$ and $q=4$]{GKT07}, applied, as in the proof of Lemma~\ref{lem:aniso:annulus}, at the top center point $(0,0)$ to the pair of footnote~\ref{foot:aniso:pressure}, whose bounded force lies in the parabolic Morrey space $M^{2,2}$ of~\cite[Equation~(6)]{GKT07}, therefore bounds $u$ on a backward cylinder at the origin: $(0,0)$ is a regular point.
\end{proof}

\begin{proof}[Proof of Theorem~\ref{thm:aniso:main}]
Proposition~\ref{prop:aniso:small} with $\rho=1/2$ gives~\eqref{eq:aniso:closure:small} with $\rho_0=1/2$, and Lemma~\ref{lem:aniso:closure} makes $(0,0)$ a regular point.
\end{proof}

Lastly, we make precise a comment that we alluded to in Remark~\ref{rem:aniso:comments}(c).

\begin{remark}[{\bf The meridional bounds suffice}]
\label{rem:interior:meridional}
Theorem~\ref{thm:aniso:main} remains valid if~\eqref{eq:aniso:bounds} is imposed on $u_r$ and $u_z$ only, with no bound on $u_\theta$ or on its derivatives. Indeed, Lemma~\ref{lem:aniso:closure} does not use~\eqref{eq:aniso:bounds}, and Section~\ref{sec:aniso:zoom} uses the bounds on $u_\theta$ in one place only, the estimate~\eqref{eq:aniso:zoom:finite:source:b}, whose role is to bound the swirl source $F_n=S_n^2/R^2$ near the axis; the $S_n$ summand of~\eqref{eq:aniso:zoom:derivatives} is used nowhere else. We replace it as follows. Since $u_r$ vanishes on the axis, the bound on $\p_ru_r$ in~\eqref{eq:aniso:bounds} gives $|u_r/r|\le\Caxi(-t)^{-1}$ on $B(1)$, so that for a fixed $N\ge\Caxi$ the function $(-t)^Nu_\theta$ solves the equation displayed in the proof of Lemma~\ref{lem:aniso:closure}, with $N$ in place of $\eta$, in which the zeroth order coefficient $r^{-2}+u_r/r+N(-t)^{-1}$ is at least $r^{-2}$, and the source is bounded by $M_f$. Moreover, $u$ is bounded on $\partial B(R_*)\times(t_*,0)$, since the radius $R_*$ and the time $t_*$ of Section~\ref{sec:aniso:zoom} are those of Lemma~\ref{lem:aniso:annulus}. Then, Lemma~\ref{lem:aniso:axis}, applied as in the proof of Lemma~\ref{lem:aniso:closure} with $t_*$ in place of $t_\eta$, gives $|u_\theta|\le C_N(-t)^{-N}$ on $B(R_*)\times(t_*,0)$, with $R_*$ and $t_*$ as in Section~\ref{sec:aniso:zoom}. In the variables of \emph{Step~1} this reads $|S_n|\le K_n:=C_N(-t_n)^{1/2+h-N}$ on compact subsets of the slab $-T\le\tau\le-1$, for $n$ large, and together with~\eqref{eq:aniso:zoom:finite:source:a}, $|RS_n|\le D_n:=C_\Gamma\delta_n$, it gives, for every $L>0$,
\begin{align}
\int_0^LF_nR^3\,dR
\le\int_0^L\min\{K_n^2R,D_n^2/R\}\,dR
&\le D_n^2\Big(\tfrac12+\log_+\frac{LK_n}{D_n}\Big)
\notag\\
&\le C_{L,N}(-t_n)^{2h}\big(1+|\log(-t_n)|\big)
\to 0
\,,
\label{eq:interior:source:lone}
\end{align}
where $\log_+a=\max\{0,\log a\}$ and $C_{L,N}$ depends also on $C_\Gamma$. Thus $F_n\to0$ in $L^1_{\rm loc}$ of the lifted variables of~\eqref{eq:aniso:zoom:lifted}, which is all that the passage to the limit in $\p_ZF_n$ requires, while the compactness argument for $\Omega_n$ in \emph{Step~2} is carried out off the axis, where $F_n\le D_n^2/R^4$ is bounded. The rest of the proof is unchanged.
\end{remark}

\appendix
\section{Properties of the OpenAI construction}
\label{app:scales}

The OpenAI manuscript~\cite{OpenAIManuscript} constructs a solution of the forced Navier-Stokes equations which it asserts to be singular at $(x,t)=(0,1)$ in its own time variable; that is, at $(0,0)$ after the time shift in this paper~\eqref{eq:time:shift}. We do not verify the construction in~\cite{OpenAIManuscript}. Instead, this appendix deduces, from statements printed in the OpenAI manuscript, three properties of the velocity field of the OpenAI construction. Claim~\ref{claim:core} is the exact axisymmetry of the velocity on a shrinking core, which is hypothesis~\eqref{eq:interior:core} of Theorem~\ref{thm:main}. Claim~\ref{claim:mean} is the set of anisotropic bounds on the angular mean, which is hypothesis~\eqref{eq:interior:mean:bounds}. Claim~\ref{claim:exterior} is the exact vanishing of the meridional velocity on a fixed open set together with the values on the axis, the two properties Remark~\ref{rem:exterior:nonanalytic} asks of a solution. None of these claims is printed in the OpenAI manuscript in the form in which we state it: each is deduced from definitions and statements printed there, and the justification following each claim gives that deduction, citing every ingredient used. We write that the manuscript \emph{states} a fact only when it is printed at the place cited, and that a fact \emph{follows} when we deduce it.

\subsection{Coordinates, notation, and the decomposition of the velocity}
\label{sec:properties:setup}

\noindent\emph{Time, the exponents, and the similarity coordinates.} The OpenAI manuscript writes $\tau=1-t$ for the time to the singularity in its own time variable. After the shift~\eqref{eq:time:shift} the time to the singularity is $-t$, and with $x$ written in the cylindrical coordinates of Section~\ref{sec:aniso:notation} we set
\begin{subequations}
\label{eq:properties:coordinates}
\begin{align}
&\tau:=-t\,,\qquad
x=(r\cos\theta,r\sin\theta,z)\,,
\notag\\
&A:=\frac12+h\,,\qquad
D:=\frac12-h\,,\qquad
A+D=1\,,\qquad
0<h<\frac1{100}
\,.
\label{eq:properties:exponents}
\end{align}
Every formula written in $\tau$ below therefore reads the same in the time variable of the paper and in that of the manuscript, whereas statements quoted from the manuscript are in its own time variable. The exponent $h$ is fixed by the construction, and $A$, $D$, and the bound on $h$ are the notation of~\cite[p.~7, Theorem~3.1, p.~15, and Theorem~4.6, p.~32]{OpenAIManuscript}. The manuscript works in the similarity coordinates
\begin{align}
q-z^2q^{2h}=\tau\,,\qquad
\eta=\frac{z}{q^{D}}\,,\qquad
X=\frac{r^2}{2q}\,,\qquad
q>0\,,\qquad
-1<\eta<1
\,,
\label{eq:properties:similarity}
\end{align}
which satisfy, with a constant $C_0$ depending only on the parameters of the construction,
\begin{align}
\tau\le q\le C_0\bigl(\tau+|z|^{1/D}\bigr)
\,,\qquad
q\asymp\tau+|z|^{1/D}
\,.
\label{eq:properties:q}
\end{align}
\end{subequations}
The manuscript states $\tau=q(1-\eta^2)$, $z=q^D\eta$, and $X=r^2/(2q)$ at~\cite[equation~(4.1), p.~24]{OpenAIManuscript}, with the ranges $q>0$ and $-1<\eta<1$ at~\cite[equation~(3.2), p.~7]{OpenAIManuscript}; it states at~\cite[p.~7]{OpenAIManuscript} that $q-z^2q^{2h}=\tau$, that this equation determines $q=q(z,\tau)$ uniquely, and that $q\asymp\tau+|z|^{1/D}$; and it states the upper bound of~\eqref{eq:properties:q} with the constant $C_0$ at~\cite[equation~(10.3), p.~118]{OpenAIManuscript}. Since $|\eta|<1$, the relation $\tau=q(1-\eta^2)$ gives $q\ge\tau$; this inequality is printed at~\cite[p.~114, Step~1]{OpenAIManuscript}. Note that $q$ and $\eta$ do not depend on $r$. We write $\mathcal C_\tau$ for the region which the manuscript calls its core, that is, the set of points whose similarity coordinates at the time $-\tau$ satisfy $0\le X\le X_c$ and $|\eta|\le\eta_c$, with the constants $X_c>0$ and $0<\eta_c<1$ fixed at~\cite[p.~8]{OpenAIManuscript}.

\smallskip
\noindent\emph{The local field and the final field.} The OpenAI manuscript first constructs a local velocity $u_{\mathrm{loc}}$ and a local pressure $\pi_{\mathrm{loc}}$ on the domain $\Omega_*$, in the form
\begin{align}
\Omega_*:=\{(x,t)\colon\tau>0\,,\ q<q_*\}
\,,\qquad
u_{\mathrm{loc}}=\curl\mathbf A+\mathbf Be_\theta
\qquad\mbox{on }\Omega_*
\,,
\label{eq:properties:local}
\end{align}
where $q_*>0$ is a constant of the construction, $\mathbf A$ is a vector potential, and $\mathbf B$ is a scalar independent of $\theta$~\cite[Theorem~3.1 and its part~(i) with equation~(3.3), p.~15, and equation~(9.21), p.~115]{OpenAIManuscript}. The manuscript writes these two fields as $\mathcal A$ and $\mathcal B$, and the second of those letters is the drift of Section~\ref{sec:aniso:zoom} in the paper. The final velocity and pressure are
\begin{align}
u=\curl(\chi_x\chi_t\mathbf A)+\chi_x\chi_t\mathbf Be_\theta
=\chi_x\chi_tu_{\mathrm{loc}}+\nabla(\chi_x\chi_t)\times\mathbf A
\,,\qquad
\pi=\chi_x\chi_t\pi_{\mathrm{loc}}
\,,
\label{eq:properties:localization}
\end{align}
where $\chi_x$ is an axisymmetric cutoff in space and $\chi_t$ is a cutoff in time, both equal to one near $(0,1)$~\cite[equation~(10.4) and the display below it, p.~118, and Section~3.5, p.~16]{OpenAIManuscript}. Accordingly, we fix once and for all $R_0>0$ and $\tau_0>0$ such that
\begin{align}
\chi_x=1\quad\mbox{on }B(R_0)
\,,\qquad
\chi_t=1\quad\mbox{for }0<\tau<\tau_0
\,,\qquad
C_0\bigl(\tau_0+R_0^{1/D}\bigr)<\frac{q_*}{2}
\,.
\label{eq:properties:cylinder}
\end{align}
Since $\chi_x\chi_t$ is then identically one on $B(R_0)\times(-\tau_0,0)$, its gradient vanishes there, so that the final fields agree with the local ones by~\eqref{eq:properties:localization}: $(u,\pi)=(u_{\mathrm{loc}},\pi_{\mathrm{loc}})$ on $B(R_0)\times(-\tau_0,0)$. By~\eqref{eq:properties:q} that cylinder lies in $\{q<q_*/2\}$, hence in $\Omega_*$. Such a choice is possible: the manuscript states at~\cite[p.~118]{OpenAIManuscript} that $\chi_x=1$ for $r\le r_0/2$ and $|z|\le z_0/2$, that $\chi_t=1$ for $0\le1-t\le\tau_*/2$, and that $C_0(\tau_*+z_0^{1/D})<q_*/2$, where $r_0$, $z_0$, and $\tau_*$ are the positive constants chosen on that page, the manuscript writing $\tau_0$ for the one we call $\tau_*$; any $R_0\le\min\{r_0,z_0\}/2$ and any $\tau_0\le\tau_*/2$ will then do, since $1/D>0$. All the norms below are spatial $L^\infty$ norms over $B(R_0)$ of the final fields at time $-\tau$, with $0<\tau<\tau_0$, and constants may depend on $R_0$ and $\tau_0$.

\smallskip
\noindent\emph{The three radial regions.} The radial thresholds of the OpenAI construction are
\begin{align}
0<X_a<X_b
\,,\qquad
0<X_a<X_{\mathrm{ext}}<\infty
\,,
\label{eq:properties:regions}
\end{align}
the first pair produced in~\cite[Theorem~4.6, pp.~32--33]{OpenAIManuscript}, whose part~(ii) states that the stress profile $T_0$ vanishes for $X\le X_a$ and for $X\ge X_b$ and is nonzero at every $X_a<X<X_b$, and the second pair in~\cite[Theorem~3.1, p.~15]{OpenAIManuscript}, with $X_{\mathrm{ext}}$ fixed beyond the supports of all the corrections in~\cite[p.~116, Step~4]{OpenAIManuscript}. We write \emph{exterior region} for $\{X\ge X_{\mathrm{ext}}\}$. The manuscript states that in the exterior region the momentum residual of the local pair $(u_{\mathrm{loc}},\pi_{\mathrm{loc}})$ vanishes identically and that
\begin{align}
\mathbf A=0
\,,\qquad
\mathbf B=K(r,\tau)
\,,\qquad
K(r,\tau)=r^{-1-2h}H_{\mathrm{ext}}(\tau/r^2)
\,,
\label{eq:exterior:source:heat}
\end{align}
where $K$ solves the radial swirl heat equation and $H_{\mathrm{ext}}$ is smooth on $[0,\infty)$ with every fixed derivative bounded~\cite[equation~(3.5), p.~15]{OpenAIManuscript}; the same two identities are stated at~\cite[p.~116, Step~4]{OpenAIManuscript}. Since $\mathbf A$ vanishes on the open set $\{X>X_{\mathrm{ext}}\}$, it follows from~\eqref{eq:properties:local} that $u_{\mathrm{loc}}=K(r,\tau)e_\theta$ there; the manuscript states as much in words at~\cite[p.~119]{OpenAIManuscript}.

\smallskip
\noindent\emph{The decomposition of the velocity.} We write $\mathcal P$ for the angular mean of~\eqref{eq:interior:average}, which averages the cylindrical coefficients in $\theta$; it is the operator $\langle\cdot\rangle_\theta$ of the OpenAI manuscript~\cite[equation~(3.7), p.~17]{OpenAIManuscript}, which commutes with the evaluation of the auxiliary variable because that evaluation depends on $(r,t)$ alone~\cite[equation~(6.3), p.~63, and p.~117]{OpenAIManuscript}. Let $u_B$ denote the background velocity of the OpenAI construction, which the manuscript states to be part of a pair of ``smooth axisymmetric fields'' $(u_B,p_B)$ with $\div u_B=0$~\cite[Proposition~5.5, p.~60]{OpenAIManuscript}, and let
\begin{align}
v:=\mathcal Pu
\,,\qquad
w:=u-v
\,,\qquad
m:=v-u_B
\,,\qquad\mbox{so that}\qquad
u=u_B+m+w
\,,\qquad
\mathcal Pw=0
\,.
\label{eq:properties:decomposition}
\end{align}
Thus $v$ and $w$ are the angular mean and the non-axisymmetric part of Section~\ref{sec:aniso:notation}, and $m$ is the correction which separates the angular mean from the background. Since $u_B$ is axisymmetric, $\mathcal Pu_B=u_B$ and $\mathcal Pm=m$. The manuscript builds $u_{\mathrm{loc}}$ from $u_B$ and corrections of two kinds, and we use its names for them: the mean corrections, whose cylindrical coefficients are independent of $\theta$~\cite[Definition~6.4, p.~69]{OpenAIManuscript}, and the wave corrections, which carry the oscillatory phases~\cite[equation~(6.27) and Definition~6.5, p.~70]{OpenAIManuscript} and whose angular mean vanishes~\cite[equation~(9.7), p.~106]{OpenAIManuscript}. Both are supported in the closed annulus $X_a\le X\le X_b$ of~\eqref{eq:properties:regions}, and the manuscript calls them annular~\cite[pp.~115 and~117]{OpenAIManuscript}.

\subsection{Exact axisymmetry on a shrinking core}
\label{sec:properties:structure}

The first claim is hypothesis~\eqref{eq:interior:core} of Theorem~\ref{thm:main}, with $\rho(t)=c\sqrt{-t}$; see item~(c) of Remark~\ref{rem:main:comments}.

\begin{claim}[{\bf Exact axisymmetry on a shrinking core}]
\label{claim:core}
Let $R_0,\tau_0$ be as in~\eqref{eq:properties:cylinder}, let $X_a$ be as in~\eqref{eq:properties:regions}, and let
\begin{align}
0<c<\sqrt{2X_a}
\,,\qquad
\tau_0(c):=\min\{\tau_0,R_0^2/c^2\}
\,.
\label{eq:properties:core:parameters}
\end{align}
Then the velocity $u$ of the OpenAI construction, its angular mean $v$, and the background velocity $u_B$ of~\eqref{eq:properties:decomposition} satisfy
\begin{align}
u=v=u_B
\qquad\mbox{on }B(c\sqrt\tau)\times\{-\tau\}
\,,\qquad
0<\tau<\tau_0(c)
\,.
\label{eq:properties:core}
\end{align}
Moreover $w(\cdot,t)=0$ on $B(c\sqrt{-t})$ for every $t\in(-1,0)$, which is~\eqref{eq:interior:core} with $\rho(t)=c\sqrt{-t}$.
\end{claim}

\begin{proof}[Justification of Claim~\ref{claim:core}]
The ingredients are the following statements of the OpenAI manuscript: the mean coefficients are independent of $\theta$ and are supported in $X_a\le X\le X_b$~\cite[Definition~6.4 and equation~(6.24), p.~69]{OpenAIManuscript}; the wave amplitudes are independent of $\theta$ and are supported in the same shell~\cite[equations~(6.28)--(6.29) and Definition~6.5, p.~70]{OpenAIManuscript}; the wave velocity increments are curls of potentials supported in that shell, and the mean velocity correction is the sum of a curl of an azimuthal potential, built through the compactly supported radial primitive $I_c$ of~\cite[equation~(8.4) with Lemma~8.2, p.~90]{OpenAIManuscript}, which vanishes for $X\le X_a$ and for $X\ge X_b$, and of a direct azimuthal field whose coefficient is supported in the same shell~\cite[Definition~9.4 and equation~(9.9), p.~106]{OpenAIManuscript}, so that ``all annular representatives vanish in a neighborhood of the axis at each $t<1$''~\cite[p.~115]{OpenAIManuscript}; and the background $(u_B,p_B)$ is a pair of smooth axisymmetric fields~\cite[Proposition~5.5, p.~60]{OpenAIManuscript}.

Let $0<\tau<\tau_0(c)$ and $|x|<c\sqrt\tau$. Since $\tau<R_0^2/c^2$ we have $c\sqrt\tau<R_0$, so that $x\in B(R_0)$ and $u=u_{\mathrm{loc}}$ at $(x,-\tau)$ by~\eqref{eq:properties:cylinder}. Moreover $r\le|x|<c\sqrt\tau$, and $q\ge\tau$ by~\eqref{eq:properties:q}, so that
\begin{align*}
X=\frac{r^2}{2q}<\frac{c^2\tau}{2q}\le\frac{c^2}{2}<X_a
\,,
\end{align*}
the last inequality by~\eqref{eq:properties:core:parameters}. At such points every mean correction and every wave correction vanishes, hence $m=w=0$ in~\eqref{eq:properties:decomposition} and the velocity is the background $u_B$, which is axisymmetric. This proves~\eqref{eq:properties:core}.

The restriction $\tau<\tau_0(c)$ enters only through the equality $u=u_{\mathrm{loc}}$, and the vanishing of $w$ does not need it. Let $t\in(-1,0)$ and $\tau=-t$. The computation above gives $X<X_a$ on $B(c\sqrt\tau)$ for every such $\tau$, so that at the points of that ball which lie in $\Omega_*$ the potentials $\mathbf A$ and $\mathbf B$ are those of the background and are independent of $\theta$. By~\eqref{eq:properties:localization} the final velocity on that ball is the curl of $\chi_x\chi_t\mathbf A$ plus $\chi_x\chi_t\mathbf Be_\theta$, where $\chi_x$ is axisymmetric, being a smooth function of $r^2$ and $z$~\cite[p.~118]{OpenAIManuscript}, $\chi_t$ is a function of $t$, and the support of $\chi_x\chi_t$ lies in $\Omega_*$~\cite[p.~118]{OpenAIManuscript}; a curl of an axisymmetric field is axisymmetric, and the final velocity vanishes off the support of $\chi_x\chi_t$. Hence $w(\cdot,t)=0$ on $B(c\sqrt{-t})$ at every $t\in(-1,0)$.

The OpenAI manuscript states ``all the annular corrections vanish near the axis for each fixed $t<1$''~\cite[p.~117]{OpenAIManuscript}, ``the annular corrections vanish in the fixed inner similarity region''~\cite[p.~124]{OpenAIManuscript}, and, at one fixed radius $X_{\mathrm{in}}\in(0,X_a)$, ``All annular corrections vanish there''~\cite[p.~116, Step~5]{OpenAIManuscript}. The vanishing on the whole set $\{X<X_a\}$, and with it the axisymmetry of the final velocity on the core, is what we deduce above.
\end{proof}

\subsection{The anisotropic bounds on the angular mean}
\label{sec:properties:values}

The second claim is hypothesis~\eqref{eq:interior:mean:bounds} of Theorem~\ref{thm:main}, with the exponent $h$ of the OpenAI construction; see item~(b) of Remark~\ref{rem:main:comments}. Theorem~\ref{thm:main} allows $0<h<1/2$, so the range $0<h<1/100$ of~\eqref{eq:properties:exponents} is admissible in Theorem~\ref{thm:main}.

\begin{claim}[{\bf Anisotropic bounds on the angular mean}]
\label{claim:mean}
Let $A,D,h$ be as in~\eqref{eq:properties:exponents}, let $R_0,\tau_0$ be as in~\eqref{eq:properties:cylinder}, and let $\beta_r=\frac12$ and $\beta_\theta=\beta_z=A$. There is a constant $C>0$, depending on $R_0$, on $\tau_0$, and on the parameters of the OpenAI construction (but not on $\tau$), such that the cylindrical components $v_j$ of the angular mean $v=\mathcal Pu$ of the velocity constructed in~\cite{OpenAIManuscript} satisfy
\begin{align}
\norm{\p_r^a\p_z^bv_j(\cdot,-\tau)}_{L^\infty(B(R_0))}
\le C\tau^{-\beta_j-a/2-bD}
\,,\qquad
j\in\{r,\theta,z\}
\,,\quad a+b\le2
\,,\quad 0<\tau<\tau_0
\,.
\label{eq:properties:mean:derivatives}
\end{align}
Since $D=\frac12-h$, these are the bounds~\eqref{eq:interior:mean:bounds} on the cylinder $B(R_0)\times(-\tau_0,0)$ (and hence on $\mathcal Q$ after the scaling of Section~\ref{sec:properties:criteria}), with the lengths $\ell_r=\tau^{1/2}$ and $\ell_z=\tau^{1/2-h}$ of~\eqref{eq:aniso:lengths}.
\end{claim}

\begin{proof}[Justification of Claim~\ref{claim:mean}]
The OpenAI manuscript states its bounds separately for the background and for the corrections of each stage, and we assemble from them a bound on the angular mean of the final velocity. By~\eqref{eq:properties:cylinder} we have $u=u_{\mathrm{loc}}$ on $B(R_0)\times(-\tau_0,0)$, and by~\eqref{eq:properties:decomposition} we have $v=u_B+m$ there; we bound $u_B$ and $m$ in turn. Two reductions serve both steps.

First, where the Cartesian derivatives of a field are bounded uniformly in $\tau$, so are the cylindrical derivatives in~\eqref{eq:properties:mean:derivatives}. The frame $(e_r,e_\theta,e_z)$ depends on $\theta$ alone, hence is constant along radial rays, so that $(e_r\cdot\nabla)e_j=0$ and
\begin{align*}
\p_r^a\p_z^bg_j=\bigl((e_r\cdot\nabla)^a\p_z^bg\bigr)\cdot e_j
\qquad\mbox{at every point with }r>0
\end{align*}
for every smooth vector field $g$ and every $j\in\{r,\theta,z\}$. The right side is a combination, with coefficients the components of the unit vector $e_r$, of the Cartesian derivatives of $g$ of order $a+b$; since the axis is a null set, this gives $\|\p_r^a\p_z^bg_j\|_{L^\infty(S)}\les\|g\|_{C^{a+b}(S)}$, the axis included, on every set $S\subset B(R_0)$ which is invariant under rotations about the axis. The same bound holds for $\p_r^a\p_z^b(\mathcal Pg)_j$, since $(\mathcal Pg)_j$ is the average of $g_j$ in $\theta$ (see Section~\ref{sec:aniso:notation}), and this average commutes with $\p_r$ and $\p_z$. Since $q$ is a function of $(z,\tau)$ alone, the sets $B(R_0)\cap\{q<c\}$ and $B(R_0)\cap\{q\ge c\}$, $c>0$, to which we apply this below, are of that kind. Since the norm of $C^{a+b}$ does not distinguish $\p_r$ from $\p_z$, and since a Cartesian component mixes $g_r$ with $g_\theta$, whose exponents $\beta_r$ and $\beta_\theta$ differ, the rates of~\eqref{eq:properties:mean:derivatives} do not pass through this reduction: we use it only for fields whose Cartesian derivatives of order at most two are bounded uniformly in $\tau$, and we obtain the bounds with rates in Steps~1 and~2 directly for the cylindrical components of $u_B$ and of $m$.

Second, the bounds quoted below are stated in the OpenAI manuscript as $q\downarrow0$, and we must cover the whole cylinder. By~\eqref{eq:properties:cylinder} we have $q<q_*/2$ on $B(R_0)\times(-\tau_0,0)$, and for every $q_1\in(0,q_*/2)$ the manuscript states that every Cartesian space-time derivative of $\mathbf A$, $\mathbf Be_\theta$, and $\pi_{\mathrm{loc}}$ is bounded up to the blowup time on compact spatial subsets with $q_1\le q\le q_*/2$~\cite[Theorem~3.1(ii), p.~15]{OpenAIManuscript}, hence so is every such derivative of $u$ by~\eqref{eq:properties:local} and~\eqref{eq:properties:cylinder}, and therefore, by the first reduction, so is every $\p_r^a\p_z^bv_j$. Since $\tau^{-\beta_j-a/2-bD}\ge\tau_0^{-\beta_j-a/2-bD}$ for $0<\tau<\tau_0$, a bounded quantity obeys~\eqref{eq:properties:mean:derivatives}, so that~\eqref{eq:properties:mean:derivatives} holds where $q\ge q_1$. Only the region $q<q_1$ remains, and Step~2 fixes the value of $q_1$.

\smallskip
\noindent\emph{Step 1: the background.} The OpenAI manuscript states the rates of its leading field $u^{(0)}$ on the region $\mathcal C_\tau$, $\|u^{(0)}_\theta\|\asymp\|u^{(0)}_z\|\asymp\tau^{-A}$ and $\|u^{(0)}_r\|=O(\tau^{-1/2})$~\cite[p.~8]{OpenAIManuscript}, and, on a fixed radial enlargement of the closed annulus $X_a\le X\le X_b$ and as $q\downarrow0$, that the cylindrical components $u_{r,B},u_{\theta,B},u_{z,B}$ of $u_B$ are such that $q^Au_{\theta,B}$, $q^Au_{z,B}$, and $q^{1/2}u_{r,B}$ differ from their leading profiles by $O(q^{2h})$ in every fixed composition of the derivatives $q\p_q,\p_X,\p_\eta$~\cite[equation~(5.42), p.~60]{OpenAIManuscript}, those profiles being bounded on that enlargement together with every fixed derivative in $(X,\eta)$~\cite[Theorem~4.6(i), p.~33]{OpenAIManuscript}. For the radial component we use only the last line of~\cite[equation~(5.42), p.~60]{OpenAIManuscript}, $q^Au_{r,B}=O(q^h)$ in the same sense, which gives $\beta_r=\frac12$ since $A-h=\frac12$. By~\eqref{eq:properties:similarity}, on functions of $(q,X,\eta)$ we have $\p_r=(2X/q)^{1/2}\p_X$ and $\p_z=q^{-D}L^{-1}\bigl(2\eta q\p_q-2\eta X\p_X+(1-\eta^2)\p_\eta\bigr)$, where $L=1-2h\eta^2\ge1-2h$, so that the conversion to physical derivatives of a cylindrical component costs $q^{-1/2}$ for each $\p_r$ and $q^{-D}$ for each $\p_z$~\cite[equation~(5.26), p.~55]{OpenAIManuscript}. Between the axis and the enlarged annulus the same bounds follow, upon using $r^2=2Xq$, from the expansion~\cite[equation~(5.1), p.~46]{OpenAIManuscript}, in whose term of order $n$ each of $q^{A-2nh}u_{z,n}$, $q^{A+1/2-2nh}u_{\theta,n}/r$, and $q^{1-2nh}u_{r,n}/r$ is a profile smooth in $(X,\eta)$ up to $X=0$~\cite[Theorem~4.6(i), p.~33, and equation~(5.17), p.~52]{OpenAIManuscript}, and from the tail~\cite[equations~(5.34)--(5.35), pp.~57--58]{OpenAIManuscript}, which past a large enough order is bounded in every Cartesian derivative of order at most two and is covered by the first reduction. In the exterior region $X\ge X_{\mathrm{ext}}$ the corrections vanish, since $X_{\mathrm{ext}}$ lies beyond their supports, so that $u_B=u_{\mathrm{loc}}=K(r,\tau)e_\theta$ by~\eqref{eq:exterior:source:heat}, and the bounds follow from~\cite[equation~(3.5), p.~15]{OpenAIManuscript} with~\cite[equations~(10.7)--(10.8), p.~119]{OpenAIManuscript}, the radial derivatives of $K$ coming from~(10.7), $\p_zK$ being zero, and $r^2\ge2X_{\mathrm{ext}}q$ there. In~\cite[Proposition~5.5, p.~60]{OpenAIManuscript} the outer endpoint of the enlargement is arbitrary, the constants there depending on it, and we fix it beyond $X_{\mathrm{ext}}$; the three regions then cover every point of $B(R_0)$. Every power of $q$ so produced has exponent at least $-\beta_j-a/2-bD<0$. Since $\tau\le q<q_1\le1$ by~\eqref{eq:properties:q}, each such power is at most $\tau^{-\beta_j-a/2-bD}$, and these statements together give~\eqref{eq:properties:mean:derivatives} for $u_B$ on $\{q<q_1\}$.

\smallskip
\noindent\emph{Step 2: the correction.} The OpenAI manuscript builds its local field by an induction whose steps it calls stages, and it states that at each finite stage the normalized velocity is the sum of the background of~\cite[Proposition~5.5, p.~60]{OpenAIManuscript}, held fixed throughout the construction~\cite[p.~62]{OpenAIManuscript}, of a mean correction whose coefficients are independent of $\theta$, and of a wave whose angular mean vanishes~\cite[equation~(9.7) and Definition~9.4, p.~106, with Definition~6.4, p.~69]{OpenAIManuscript}. The stages are summed in~\cite[Lemma~5.4 and equations~(5.34)--(5.35), pp.~57--58]{OpenAIManuscript}, as applied in~\cite[Proposition~9.9, p.~114]{OpenAIManuscript}. The summation cutoffs are of the form $\chi(a_jq)$, functions of $q(z,\tau)$ alone, with $\chi$ equal to one on $[0,1/2]$ and $a_{j+1}\ge2a_j$~\cite[equation~(9.21), p.~115, and Lemma~5.4, p.~57]{OpenAIManuscript}, so that for each stage $J$ there is $q_J>0$ with $\chi(a_jq)=1$ for every $j\le J$ and every $0<q<q_J$; the single cutoff applied to the fields with which the summation starts likewise equals one near $q=0$~\cite[p.~114, Step~1, and p.~115]{OpenAIManuscript}. The tail estimate~\cite[equation~(5.35), p.~58]{OpenAIManuscript} then gives, in every fixed order of Cartesian derivatives, that on $\{q<q_J\}$ the summed field differs from the stage-$J$ field by $O(q^{N_J})$, with $N_J\to\infty$ as $J\to\infty$. We now fix $J$ so large that $N_J\ge0$ for every derivative order at most two, and we set $q_1=q_J$, decreased if necessary so that $q_1<\min\{1,q_*/2\}$, so that the single cutoff above equals one on $\{q<q_1\}$, and so that the statements quoted in Step~1, which hold for small $q$, hold on $\{q<q_1\}$. Applying $\mathcal P$ to the stage-$J$ field, and using $\mathcal Pu_B=u_B$ from~\eqref{eq:properties:decomposition}, we obtain that on $\{q<q_1\}$ the correction $m$ is the mean correction of that stage, up to the angular mean of an error whose Cartesian derivatives of order at most two are $O(q^{N_J})$, hence bounded; by the first reduction, this term obeys~\eqref{eq:properties:mean:derivatives}.

It remains to bound the mean correction of a fixed stage. The OpenAI manuscript works on dyadic regions indexed by an integer $\ell$, with
\begin{align*}
Q=2^{-\ell}
\,,\qquad
\eps=Q^h
\,,\qquad
S_*=\ell^2
\,,\qquad
\kappa_s=10^{-5}
\,,
\end{align*}
on each of which a physical velocity is $Q^{-A}$ times its representative in the rescaled coordinates of~\cite[equation~(6.1), p.~62, and the text following it and equation~(6.2), p.~63]{OpenAIManuscript}, and $q\asymp Q$ on the supports of the corrections~\cite[pp.~62 and~113]{OpenAIManuscript}. The manuscript states that the radial component of the mean correction carries a factor $\eps^{1.9}$ and its azimuthal and axial components a factor $\eps^{0.9}$, in every coefficient derivative, times a power of $S_*$ and a weight which is bounded on the annulus $X_a<X<X_b$~\cite[equation~(9.9), p.~106, with Definition~6.4 and equations~(6.23)--(6.24), p.~69]{OpenAIManuscript}, and that a radial physical derivative loses a factor $Q^{-1/2-h\kappa_s}$ and an axial one a factor $Q^{-D}$~\cite[equation~(9.19), p.~113]{OpenAIManuscript}. With the bounded error of the previous paragraph, it follows that on $\{q<q_1\}$, on the annulus and trivially off it since the mean coefficients vanish there,
\begin{align}
\abs{\p_r^a\p_z^bm_j}
\le Cq^{-A-a/2-bD+(\alpha_j-a\kappa_s)h}\bigl(1+\abs{\log q}\bigr)^{P}
\,,\qquad
\alpha_r=1.9
\,,\quad
\alpha_\theta=\alpha_z=0.9
\,,\quad
a+b\le2
\,,
\label{eq:properties:mean:error}
\end{align}
where $P$ depends on $a+b$ and on the stage $J$ fixed above. The radial row of~\eqref{eq:properties:mean:error} carries $\beta_r=\frac12$ because $-A+(1.9-a\kappa_s)h=-\frac12+(0.9-a\kappa_s)h$. In every row the exponent of $q$ in~\eqref{eq:properties:mean:error} is negative, and it exceeds $-\beta_j-a/2-bD$ by $(0.9-a\kappa_s)h>0$; since $\tau\le q<q_*/2$, the powers of $q$ become powers of $\tau$ and the logarithm is absorbed. Thus $m$ satisfies~\eqref{eq:properties:mean:derivatives} on $\{q<q_1\}$, and with Step~1 so does $v=u_B+m$. Together with the second reduction above, which covers $\{q\ge q_1\}$, this proves~\eqref{eq:properties:mean:derivatives} on $B(R_0)\times(-\tau_0,0)$.
\end{proof}

\subsection{Exact vanishing of the meridional velocity, and the values on the axis}
\label{sec:exterior:source}

Remark~\ref{rem:exterior:nonanalytic} asks two things of a solution: the exact vanishing~\eqref{eq:exterior:vanishing} of the meridional components $u_r$ and $u_z$ on one open set off the axis which does not depend on time, and a time at which $u_z$ does not vanish at the origin. The third claim supplies both.

\begin{claim}[{\bf Vanishing of the meridional velocity on a fixed open set, and the values on the axis}]
\label{claim:exterior}
Let $R_0,\tau_0$ be as in~\eqref{eq:properties:cylinder}, and let $u$ be the velocity of the OpenAI construction.
\begin{enumerate}[label=(\alph*),leftmargin=2em]
\item Let $0<r_-<r_+<R_0$. There exist $z_*>0$ and $0<\delta_*\le\tau_0$ such that the open set
\begin{align*}
E:=\{x\in\RR^3\colon r_-<r<r_+\,,\ |z|<z_*\}
\end{align*}
is contained in $B(R_0)\cap\{r>0\}$, and such that
\begin{align*}
u=K(r,\tau)e_\theta
\,,\qquad\mbox{and in particular}\qquad
u_r=u_z=0
\,,\qquad\mbox{on }E\times(-\delta_*,0)
\,,
\end{align*}
where $K$ is the function in~\eqref{eq:exterior:source:heat}. This is~\eqref{eq:exterior:vanishing} with $R=R_0$ and $\delta=\delta_*$.
\item There is a constant $j_0$ with $0<j_0\le\frac1{20}$ such that
\begin{align}
u_z(0,-\tau)=j_0\tau^{-A}
\,,\qquad
\p_zu_z(0,-\tau)=\frac4\tau
\,,\qquad
0<\tau<\tau_0
\,.
\label{eq:properties:axis}
\end{align}
In particular $u_z(0,t)\ne0$ for every $t\in(-\tau_0,0)$.
\end{enumerate}
\end{claim}

\begin{proof}[Justification of Claim~\ref{claim:exterior}]
\noindent\emph{Part (a).} By~\eqref{eq:exterior:source:heat} and~\eqref{eq:properties:local} we have $u_{\mathrm{loc}}=K(r,\tau)e_\theta$, and in particular $u_{\mathrm{loc},r}=u_{\mathrm{loc},z}=0$, on the open set $\{X>X_{\mathrm{ext}}\}$. Given $0<r_-<r_+<R_0$, we choose $z_*>0$ and $0<\delta_*\le\tau_0$ so small that
\begin{align*}
E\subset B(R_0)
\,,\qquad
C_0\bigl(\delta_*+z_*^{1/D}\bigr)<\frac{r_-^2}{2X_{\mathrm{ext}}}
\,,
\end{align*}
which is possible since $r_+<R_0$ and $1/D>0$. By the upper bound in~\eqref{eq:properties:q}, every point of $E\times(-\delta_*,0)$ has $q<r_-^2/(2X_{\mathrm{ext}})$, and hence $X=r^2/(2q)>r_-^2/(2q)>X_{\mathrm{ext}}$. Since $\delta_*\le\tau_0$ and $E\subset B(R_0)$, we have $u=u_{\mathrm{loc}}$ there by~\eqref{eq:properties:cylinder}, and part~(a) follows.

\smallskip
\noindent\emph{Part (b).} The OpenAI manuscript states $u^{(0)}_\theta=0$ on the axis and $\|u^{(0)}_z\|_{L^\infty}\asymp\tau^{-A}$ on $\mathcal C_\tau$, the lower bound being attributed in the manuscript to the nonzero axis datum~\cite[pp.~7--8]{OpenAIManuscript}; the manuscript does not state the pointwise axial value or its axial derivative on the axis. It chooses the axis datum $U_*=4\eta+j_0$, with $0<j_0\le.05$, as the value at the axis of the axial profile $U$~\cite[Section~B.1 and equation~(B.1), p.~144]{OpenAIManuscript}, this being one of the axis parameters of~\cite[Proposition~4.10, pp.~36--37]{OpenAIManuscript}, and it states that the analytic profile realizing that datum has $U-U_*$ vanishing at the axis~\cite[Proposition~B.2 and equation~(B.12), pp.~146--147]{OpenAIManuscript}; that the axial coefficients of the expansion are $u_{z,n}=q^{-A+2nh}U_n$~\cite[equation~(5.1), p.~46]{OpenAIManuscript}, that these enter the background with the cutoffs of~\cite[equation~(5.45), p.~61]{OpenAIManuscript} as $q^{-A+2nh}\chi(c_nq)U_n$, the axial coefficient carrying no derivative of its cutoff because $q$ does not depend on $r$~\cite[equation~(10.1), p.~117]{OpenAIManuscript}, and that $U_n(0,\eta)=0$ for every $n\ge1$~\cite[Lemma~5.1, p.~47]{OpenAIManuscript}; and that $q=\tau$ and $\eta=0$ at $z=0$~\cite[p.~116, Step~5]{OpenAIManuscript}. Every term with $n\ge1$ in the sum of the axial coefficients therefore vanishes at $X=0$, so that
\begin{align*}
u_{z,B}=q^{-A}U(0,\eta)=q^{-A}(4\eta+j_0)
\qquad\mbox{on the axis}
\,.
\end{align*}
Since $X=0<X_a$ on the axis, $m=w=0$ there as in the justification of Claim~\ref{claim:core}, and $u=u_{\mathrm{loc}}$ by~\eqref{eq:properties:cylinder}; hence $u_z=q^{-A}(4\eta+j_0)$ on the segment $\{r=0,\ |z|<R_0\}$ of the axis, for $0<\tau<\tau_0$. Evaluating at $z=0$, where $q=\tau$ and $\eta=0$, gives the first identity in~\eqref{eq:properties:axis}. For the second, differentiating in $z$ and using $\eta=zq^{-D}$ together with the identity $\p_zq=2\eta q^{1-D}/L$, where $L=1-2h\eta^2$~\cite[equation~(4.1), p.~24, and the proof of Lemma~4.1, p.~25]{OpenAIManuscript}, so that $\p_zq=0$ at $\eta=0$, we obtain $\p_zu_z(0,-\tau)=4\tau^{-A}\tau^{-D}$, which is $4/\tau$ by~\eqref{eq:properties:exponents}.
\end{proof}

\subsection{The blowup and the non-axisymmetric part}
\label{sec:properties:waves}

\noindent\emph{The blowup.} The OpenAI manuscript asserts that, along the points
\begin{align}
x_\tau=(\sqrt{2X_{\mathrm{in}}\tau},0,0)
\,,\qquad
u_\theta(x_\tau,-\tau)=\tau^{-A}\bigl(e_0+O(\tau^{2h})\bigr)\longrightarrow+\infty
\,,
\label{eq:properties:blowup}
\end{align}
where $z=0$, $q=\tau$, and both cutoffs equal one for all sufficiently small $\tau>0$~\cite[equations~(10.20)--(10.21), p.~124]{OpenAIManuscript}, with $X_{\mathrm{in}}\in(0,X_a)$ and $e_0>0$ produced at~\cite[Theorem~3.1(iv) and equation~(3.6), p.~16]{OpenAIManuscript}. Since $\abs{x_\tau}\to0$, the point $(0,0)$ is then singular for the construction.

\smallskip
\noindent\emph{The non-axisymmetric part.} In chart variables the amplitudes of the wave corrections carry a factor $\eps^{1/2}$ relative to the background~\cite[equation~(9.9), p.~106, with Definition~6.5 and equation~(6.29), p.~70]{OpenAIManuscript}. We do not convert this into a bound for $w$ in physical variables: Theorem~\ref{thm:main} uses only the exact axisymmetry on the core of Claim~\ref{claim:core} and the bounds on the angular mean of Claim~\ref{claim:mean}. It is in this sense, exact axisymmetry on a core which collapses at the parabolic rate as $t\uparrow0$, that we call such solutions asymptotically axisymmetric.

\subsection{What the results of the paper give for the OpenAI construction}
\label{sec:properties:criteria}

The OpenAI manuscript states that its velocity and pressure are smooth on $\RR^3\times[0,1)$, divergence free, both of fixed compact spatial support~\cite[Proposition~10.1, p.~117]{OpenAIManuscript}, that its force, the momentum residual~\cite[equation~(10.5), p.~118]{OpenAIManuscript}, lies in $C^\infty_c(\RR^3\times(0,\infty);\RR^3)$~\cite[Lemma~10.3, p.~120]{OpenAIManuscript}, which gives~\eqref{eq:interior:force:c-two}, and that its kinetic energy is bounded on $[0,1)$ and its total dissipation on $[0,1)$ is finite~\cite[Lemma~10.4 and equation~(10.13), p.~121]{OpenAIManuscript}, the velocity part of~\eqref{eq:interior:energy:class}. Since $u$ and $\pi$ are compactly supported, $\pi$ is the decaying solution of the pressure equation $-\Delta\pi=\p_i\p_j(u_iu_j)-\div f$ of~\cite[p.~118]{OpenAIManuscript}, so that $\norm{\pi(\cdot,t)}_{L^{3/2}(B(1))}\les\norm{u(\cdot,t)}_{L^3}^2+1$ by the Calderon-Zygmund inequality and the boundedness of $f$; interpolation places $u$ in $L^3_{x,t}$, and the pressure part of~\eqref{eq:interior:energy:class} follows. Smoothness on $\RR^3\times[-1,0)$ makes the local energy inequality an equality, upon testing~\eqref{eq:nse:forced} against $2u\varphi$, every term being finite by~\eqref{eq:interior:energy:class}; the construction is thus a suitable weak solution on the cylinder $B(R_0)\times(-\tau_0(c),0)$ of Claim~\ref{claim:core}, carried onto $\mathcal Q$ by the scaling in the footnote of Section~\ref{sec:aniso:setup}. Claims~\ref{claim:mean} and~\ref{claim:core} supply~\eqref{eq:interior:mean:bounds} and~\eqref{eq:interior:core}, so Corollary~\ref{cor:interior:nonanalytic} applies: conditionally on the blowup~\eqref{eq:properties:blowup}, the force violates~\eqref{eq:interior:force:analytic} and vanishes identically on no cylinder around $(0,0)$. Claim~\ref{claim:exterior} gives the same failure through Remark~\ref{rem:exterior:nonanalytic}, on every ball $B(R')$ with $0<R'\le R_0$ which meets its set $E$, using neither~\eqref{eq:interior:mean:bounds} nor~\eqref{eq:interior:core} nor the blowup. Were~\eqref{eq:interior:error:bound} to hold on a cylinder $B(R')\times(-\delta',0)$ with $0<R'\le R_0$ and $0<\delta'\le\tau_0(c)$, Remark~\ref{rem:interior:approximate} would make $(0,0)$ a regular point, which would contradict~\eqref{eq:properties:blowup}; hence the non-axisymmetric part of the velocity is not bounded in $C^3$ uniformly in time on any such cylinder.

\section*{Acknowledgments}
The work of P.C. was partially supported by NSF grant DMS-2606072. The work of M.I. was partially supported by NSF grant DMS-2204614. The work of V.V. was partially supported by the Collaborative NSF grant DMS-2307681 and by a Simons Investigator Award.

We are grateful to Scott Armstrong and Tuomo Kuusi for comparing the properties of the OpenAI construction recorded in Appendix~\ref{app:scales} against the Lean code released by OpenAI.

\section*{Disclaimer on the use of AI systems} 
Publicly available versions of Claude (Opus 5 and Fable 5.1, within Claude Code) and ChatGPT (Sol and Astra, within Codex) were used in the preparation of this manuscript. By far the heaviest use of these AI systems was in deciphering the statements of the OpenAI paper~\cite{OpenAIManuscript}. Claude and ChatGPT were used to analyze that manuscript, and over the course of several interactive sessions the main properties of the solutions in the OpenAI construction were extracted; see Appendix~\ref{app:scales}. These properties were subsequently cross-checked against the publicly available PDF version of~\cite{OpenAIManuscript} and the Lean code released by OpenAI. The AI systems were given access to our paper~\cite{CIV26} and to a nearly completed version of the paper~\cite{CIVEulerLength}, from which the core idea of this proof is taken. The AI systems were also provided with an early draft of this manuscript, written entirely by the authors, which contained the entire proof of Theorem~\ref{thm:aniso:main} in the form of a six-page sketch, and a half-page explanation of how the analyticity reduction to the axisymmetric case should work. Claude and ChatGPT were used to turn this early draft into a manuscript which contained complete proofs of all statements. The AI systems also produced the figure and searched the relevant bibliography. Over the course of a few interactive sessions of Claude Code, the manuscript was then rewritten under the close guidance and supervision of the authors; this included not just the polishing of the exposition, but also the verification of all mathematical computations. Every suggestion, computation, and reference produced by the AI systems was independently verified and, where necessary, corrected or reformulated by the authors, who take full responsibility for the manuscript as it is presently written.

\end{document}